\documentclass{article}
\usepackage[utf8]{inputenc}
\usepackage{amssymb,amsthm,amsmath,tikz}
\usetikzlibrary{calc}
\usepackage{a4wide}
\usepackage{subcaption}

\newtheorem{theorem}{Theorem}
\newtheorem{conjecture}[theorem]{Conjecture}

\newtheorem{corollary}[theorem]{Corollary}

\newtheorem{lemma}[theorem]{Lemma}
\newtheorem{proposition}[theorem]{Proposition}

\newtheorem{problem}[theorem]{Problem}
\newtheorem{oproblem}[theorem]{Open Problem}

\newcommand{\LD}{\gamma^{LD}}

\usepackage{todonotes}
\usepackage{url}

\usepackage{hyperref}
\newcommand{\footremember}[2]{%
	\footnote{#2}
	\newcounter{#1}
	\setcounter{#1}{\value{footnote}}%
}
\newcommand{\footrecall}[1]{%
	\footnotemark[\value{#1}]%
}

\makeatletter
\def\thmheadcasestyle#1#2#3{%
  \thmname{#1}\thmnumber{\@ifnotempty{#1}{ }\@upn{#2}}%
  \thmnote{{\the\thm@notefont: \textbf{#3}}}}
\makeatother

\newtheoremstyle{casestyle}
  {}
  {}
  {\normalfont}
  {0pt}
  {\bfseries}
  {. \vspace{0.05em}\newline}
  { }
  {\thmheadcasestyle{#1}{#2}{#3}}
  
\makeatletter
\def\thmheadsubcasestyle#1#2#3{%
  \thmname{#1}\thmnumber{\@ifnotempty{#1}{ }\@upn{#2}}%
  \thmnote{{\the\thm@notefont: \textbf{#3}}}}
\makeatother
  
\newtheoremstyle{subcasestyle}
  {}
  {}
  {\normalfont}
  {0pt}
  {\bfseries}
  {.\vspace{0.05em}\newline}
  { }
  {\thmheadsubcasestyle{#1}{#2}{#3}}

\theoremstyle{casestyle}

\newtheorem{case}{Case}{}

\theoremstyle{subcasestyle}

\newtheorem{subcase}{$\blacktriangleright$ Case}[case]

\newtheorem{subsubcase}{$\blacktriangleright \blacktriangleright$ Case}[subcase]

\newtheorem{subsubsubcase}{$\blacktriangleright \blacktriangleright \blacktriangleright$ Case}[subsubcase]

\newtheoremstyle{claimstyle}
  {}
  {}
  {\normalfont}
  {0pt}
  {\bfseries}
  {.\vspace{0.1em}}
  { }
  {\thmheadsubcasestyle{#1}{#2}{#3}}
  
  \theoremstyle{claimstyle}

{}

\newcommand{\calG}{\mathcal{G}}

\title{The $n/2$-bound for locating-dominating sets in subcubic graphs}

\author{Dipayan Chakraborty\footnote{Université Clermont-Auvergne, CNRS, Mines de Saint-Étienne, Clermont-Auvergne-INP, LIMOS, 63000 Clermont-Ferrand, France.} \footnote{Department of Mathematics and Applied Mathematics, University of Johannesburg, South Africa.} \footnote{Department of Computer Science and Mathematics, Lebanese American University, Beirut, Lebanon} \footnote{Aix-Marseille Université, CNRS, École Centrale Méditerranée, LIS, Marseille, France} \footnote{Department of Computer Science and Engineering, SRM University-AP, Amaravati, Andhra Pradesh 522 240, India}
\and
Anni Hakanen\footnote{Turku Collegium for Science, Medicine and Technology TCSMT, University of Turku, Finland.} \footremember{TY}{Department of Mathematics and Statistics, University of Turku, FI-20014 Turku, Finland.}
\and
Tuomo Lehtil\"a\footrecall{TY} 
}
\date{}

\begin{document}

\maketitle

\begin{abstract}
The location-domination number is conjectured to be at most half of the order for twin-free graphs with no isolated vertices. We prove that this conjecture holds and is tight for subcubic graphs. We also show that the same upper bound holds for subcubic graphs with open twins of degree 3 and closed twins of any degree, but not for subcubic graphs with open twins of degree 1 or 2. These results then imply that the same upper bound holds for all cubic graphs (with or without twins) except $K_4$ and $K_{3,3}$. 

\end{abstract}

\section{Introduction}

In this article, we consider a well-known conjecture for locating-dominating sets on twin-free graphs for subcubic graphs. In particular, we prove the conjecture for subcubic graphs by giving a result stronger than the suggested conjecture. Furthermore, we answer positively to two open problems posed by Foucaud and Henning in~\cite{foucaud2016location}.

Let $G = (V(G),E(G))$ denote a graph with vertex set $V(G)$ and edge set $E(G)$. Let us denote by $N_G(v)$ and $N_G[v]$  (or $N(v)$ and $N[v]$ if  $G$ is clear from context) the \textit{open} and \textit{closed neighbourhoods}, respectively, of a vertex $v$ in a graph $G$.
We further denote by $I_G(S;v)=N_G[v]\cap S$ (or $I(v)$ if $S$ and $G$ are clear from context) the \emph{$I$-set} of $v$ in $G$ with respect to the set $S\subseteq V(G)$.
A set of vertices $S\subseteq V(G)$ is \textit{dominating} if each vertex of $G$ has a non-empty $I$-set, that is, if the vertices are  in $S$ or have an adjacent vertex in $S$. Moreover, a set  $S$ \textit{separates} vertices $u$ and $v$ if $I(u)\neq I(v)$. Similarly, we say that the set $S$ separates vertices in a set $V\subseteq V(G)$ if the $I$-set of each vertex in $V$ is unique.
Finally, a set $S$ is \textit{locating-dominating} in $G$ (LD-set for short) if $S$ is dominating in $G$ and for each pair of distinct vertices $u,v\in V(G)\setminus S$ we have $I(u)\neq I(v)$. The \textit{location-domination number} of a graph $G$, $\LD(G)$, denotes the size of a smallest locating-dominating set in $G$. The concept of location-domination was originally introduced by Slater and Rall in~\cite{rall1984location, slater1988dominating}. See the electronical bibliography~\cite{Lobstein} to find over 500 papers on topics related to location-domination.

The \textit{degree} of a vertex $v$ in a graph $G$, $\deg_G(v)$, denotes the number of vertices in $N(v)$. A graph is \emph{$r$-regular} if each vertex has the same degree~$r$. Furthermore, a $3$-regular graph is called \textit{cubic graph} and a graph in which every vertex has degree of at most~$3$ is called \textit{subcubic}. Two vertices $u,v$ are said to be \textit{open twins} if $N(u)=N(v)$ and \textit{closed twins} if $N[u]=N[v]$. A graph $G$ is \textit{open twin-free} if it has no pairs of open twins and is called \textit{closed twin-free} if it has no pairs of closed twins in it. Moreover, $G$ is called \textit{twin-free} if it is both open twin-free and closed twin-free. We denote by $n(G)$, or by $n$ if the graph is clear from context, the number of vertices in $G$ (i.e. the \textit{order} of $G$). Similarly, the number of edges is denoted by $m(G)$ or simply by $m$.

The following conjecture that we study in this paper was originally proposed in~\cite[Conjecture 2]{GGM2014} and was first considered in the formulation of Conjecture~\ref{conj_Garijo} in~\cite[Conjecture 2]{foucaud2016location}.
\begin{conjecture}[\cite{foucaud2016location,GGM2014}]\label{conj_Garijo}
Every twin-free graph $G$ of order $n$ without isolated vertices satisfies $\LD(G) \leq \frac{n}{2}$.
\end{conjecture}
\noindent Both restrictions in Conjecture~\ref{conj_Garijo} are required since every isolated vertex is required to be in a dominating set, the complete graph $K_n$ (with closed twins) has $\LD(K_n)=n-1$ and the star $K_{1,n-1}$ (with open twins) has $\LD(K_{1,n-1})=n-1$. 

Conjecture~\ref{conj_Garijo} has been studied in at least~\cite{bousquet2024note, chakraborty2023three, claverol2021metric,  foucaud2016location, foucaud2017location, foucaud2016locating, GGM2014}. Variations on directed graphs have been studied in~\cite{bellitto2023locating, bousquet2023locating, foucaud2020domination} and for locating-total-dominating sets in~\cite{chakraborty2022progress, foucaud2017location, foucaud2016locating}. See~\cite{foucaud2022problems} for a short introduction to the conjecture. 
Besides introducing the conjecture, Garijo et al.~\cite{GGM2014} gave a general upper bound $\lfloor 2n/3\rfloor+1$ for location-domination number in twin-free graphs. Recently, Bousquet et al.~\cite{bousquet2024note} improved this general upper bound to $\LD(G)\leq\lceil 5n/8\rceil$. Besides proving general upper bounds, a lot of the research on this conjecture has concentrated on proving it for some graph classes. In particular, Garijo et al.~\cite{GGM2014} proved the following useful theorem.

\begin{theorem}[\cite{GGM2014}]\label{The4cycle}
Let $G$ be a connected twin-free graph without $4$-cycles on $n\geq 2$  vertices. We have \[\LD(G)\leq \frac{n}{2}.\] 
\end{theorem}

\noindent Furthermore, Garijo et al.~\cite{GGM2014} proved the conjecture for graphs with independence number at least $\lceil n/2\rceil$ (in particular this class includes bipartite graphs) and graphs with clique number at least $\lceil n/2\rceil+1$. In~\cite{foucaud2016locating}, Foucaud et al. proved the conjecture for split and co-bipartite graphs, in~\cite{chakraborty2023three}, Chakraborty et al. proved the conjecture for block graphs and, in~\cite{foucaud2017location}, Foucaud and Henning proved the conjecture for line graphs. The conjecture has also been proven for maximal outerplanar graphs in~\cite{claverol2021metric} by Claverol et al. and for cubic graphs in~\cite{foucaud2016location} by Foucaud and Henning. In this article, we concentrate on subcubic graphs. 

Besides upper bounds, also lower bounds of locating-dominating sets have been considered in the literature. In~\cite[Theorem 2]{slater2002fault}, Slater has given a lower bound for the location-domination number of $r$-regular graphs. In particular, the result states that $\LD(G)\geq \frac{n}{3}$ for cubic graphs. However, the proof also holds for subcubic graphs in general. 

In the following, we introduce some further notations. A \textit{leaf} is a vertex of degree one and a \textit{support vertex} is a vertex adjacent to a leaf. For a vertex $v\in V(G)$ we denote the graph formed from $G$ by removing $v$ and each edge incident with $v$ by $G-v$. Similarly, for an edge $e\in E(G)$ we denote by $G-e$ the graph obtained from $G$ by removing edge $e$. For a set $D$ of vertices and/or edges of $G$, we denote by $G-D$ the graph obtained from $G$ by removing each vertex and edge in $D$. A graph is \textit{triangle-free} if each induced cycle in the graph has at least four vertices. 

A set of vertices $S$ is \textit{total dominating} in $G$ if each vertex in $V(G)$ is adjacent to another vertex in $S$. Furthermore, a set is \textit{locating-total dominating} if it is locating-dominating and total dominating.


\subsection{Our results}\label{SecResults}

Our results consider Conjecture~\ref{conj_Garijo} in subcubic graphs. Besides proving the conjecture for subcubic graphs, we also answer the following open problems posed by  Foucaud and Henning in~\cite{foucaud2016location}:

\begin{problem}[\cite{foucaud2016location}]\label{ProbSubcubic}
Determine whether Conjecture~\ref{conj_Garijo} can be proven for subcubic graphs.
\end{problem}

\begin{problem}[\cite{foucaud2016location}]\label{ProbCubicTwins}
Determine whether Conjecture~\ref{conj_Garijo} can be proven for connected cubic graphs in general (allowing twins) with the exception of a finite set of forbidden graphs.
\end{problem}

As Foucaud and Henning have noted, Problem~\ref{ProbCubicTwins} is a weaker form of a conjecture from~\cite{henning2012locating} by Henning and Löwenstein and open problem by Henning and Rad from~\cite{henningRad2012locating} stating that for a connected cubic graph $G$ the locating-total dominating number is at most half the number of vertices of $G$. In particular, this bound does not hold for subcubic graphs which might require two-thirds of vertices~\cite{chakraborty2022progress} (see Figure~\ref{Fig_totalCounter}). Moreover, as there has not been progress on this problem for over a decade, hopefully our results for locating-dominating sets can help in solving this stronger conjecture. 

\begin{figure}[!htb]
\centering
\begin{tikzpicture}[
blacknode/.style={circle, draw=black!, fill=black!, thick},
whitenode/.style={circle, draw=black!, fill=white!, thick},
scale=0.5]
\tiny
\node[whitenode] (0) at (0,0) {};
\node[whitenode] (1) at (3,0) {};
\node[whitenode] (2) at (6,0) {};
\node[whitenode] (3) at (9,0) {};
\node[whitenode] (4) at (12,0) {};
\node[whitenode] (5) at (15,0) {};
\node[whitenode] (6) at (18,0) {};
\node[whitenode] (7) at (21,0) {};
\node[blacknode] (0') at (0,2) {};
\node[blacknode] (0'') at (0,4) {};

\node[blacknode] (1') at (3,2) {};
\node[blacknode] (1'') at (3,4) {};

\node[blacknode] (2') at (6,2) {};
\node[blacknode] (2'') at (6,4) {};

\node[blacknode] (3') at (9,2) {};
\node[blacknode] (3'') at (9,4) {};

\node[blacknode] (4') at (12,2) {};
\node[blacknode] (4'') at (12,4) {};

\node[blacknode] (5') at (15,2) {};
\node[blacknode] (5'') at (15,4) {};

\node[blacknode] (6') at (18,2) {};
\node[blacknode] (6'') at (18,4) {};

\node[blacknode] (7') at (21,2) {};
\node[blacknode] (7'') at (21,4) {};

\draw[-, thick] (0) -- (1);
\draw[-, thick] (1) -- (2);
\draw[-, thick] (2) -- (3);
\draw[-, thick] (3) -- (4);
\draw[-, thick] (4) -- (5);
\draw[-, thick] (5) -- (6);
\draw[-, thick] (6) -- (7);
\draw[dashed, thick] (7) -- (23,0);
\draw[dashed, thick] (0) -- (-2,0);
\draw[-, thick] (0) -- (0');
\draw[-, thick] (0') -- (0'');

\draw[-, thick] (1) -- (1');
\draw[-, thick] (1') -- (1'');

\draw[-, thick] (2) -- (2');
\draw[-, thick] (2') -- (2'');

\draw[-, thick] (3) -- (3');
\draw[-, thick] (3') -- (3'');

\draw[-, thick] (4) -- (4');
\draw[-, thick] (4') -- (4'');

\draw[-, thick] (5) -- (5');
\draw[-, thick] (5') -- (5'');

\draw[-, thick] (6) -- (6');
\draw[-, thick] (6') -- (6'');

\draw[-, thick] (7) -- (7');
\draw[-, thick] (7') -- (7'');
\end{tikzpicture}
\caption{Example of a subcubic graph which shows that the $\frac{n}{2}$-upper bound for the locating total-dominating number of a twin-free graph on $n$ vertices is not true. The shaded vertices constitute a minimum LTD-set.}\label{Fig_totalCounter}
\end{figure}
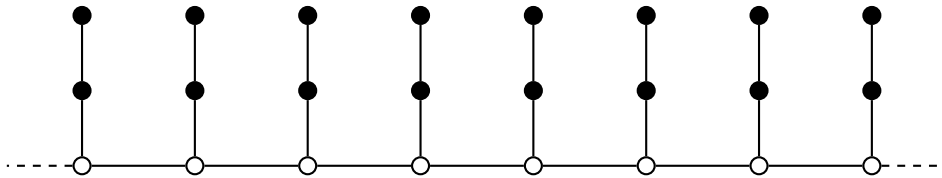

Coming back to our results, in Proposition~\ref{propSubcubic}, we show that a twin-free subcubic graph $G$ has $\LD(G)\leq \frac{n}{2}$. This answers positively to Problem~\ref{ProbSubcubic}. Then, we continue to Theorem~\ref{theclosedtwins} where we expand Proposition~\ref{propSubcubic} by allowing subcubic graphs to include closed twins. Finally, in Theorem~\ref{thm_deg3}, we expand the result also for subcubic graphs allowing open twins of degree~$3$ with the exceptions for exactly complete graph $K_4$ and complete bipartite graph $K_{3,3}$. This answers positively to Problem~\ref{ProbCubicTwins}. Observe that we actually give a result stronger than the Problems~\ref{ProbCubicTwins} and~\ref{ProbSubcubic} asked. We state that the $\frac{n}{2}$ bound holds for connected subcubic graphs without open twins of degrees 1 or 2 on at least four vertices with the exceptions of $K_4$ and $K_{3,3}$.

We also show in Proposition~\ref{PropOpentwinEx} that forbidding open twins of degrees~$1$ and~$2$ is necessary. Furthermore, in Proposition~\ref{PropRegEx}, we show that Problem~\ref{ProbCubicTwins} cannot be expanded for $r$-regular graphs, for $r\geq4$, with twins. We also give an infinite family of twin-free subcubic graphs for which the conjecture is tight in Proposition~\ref{PropTightSubcubic}. Note that Foucaud and Henning had asked in~\cite[Problem 1]{foucaud2016location} for characterizing each twin-free cubic graph which attain the $\frac{n}{2}$-bound. We have manually checked every twin-free 10-vertex cubic graph and found that there were  no tight examples. Furthermore, we are aware of only one 6-vertex and one 8-vertex twin-free cubic graph (presented in~\cite{foucaud2016location}) which attain this upper bound. From this perspective, the proof for subcubic graphs seems more challenging than the one for cubic graphs.


\subsection{Structure of the paper}

The paper is structured as follows: In Subsection~\ref{subsecLem}, we introduce some useful lemmas which are used throughout this paper. Then, in Section~\ref{secMain}, we prove Conjecture~\ref{conj_Garijo} for subcubic graphs allowing closed twins. In Section~\ref{SecOpentwins}, we prove the conjecture for subcubic graphs allowing open twins of degree $3$. We continue in Section~\ref{SecExample}, by giving some constructions which show tightness of our results and show that they cannot be further generalized. Finally, we conclude in Section~\ref{SecConclusion}.


\subsection{Lemmas}\label{subsecLem}

Lemmas~\ref{LemLDSup} and~\ref{LemLDLeaves} have previously been considered for trees in~\cite{blidia2007locating} and~\cite{bousquet2023locating}. The proofs for these generalizations follow similarly as the original ones but we offer them for the sake of completeness.

\begin{lemma}\label{LemLDSup}
Let $G$ be a connected graph on at least three vertices. Then $G$ admits  an optimal locating-dominating set $S$ such that every support vertex of $G$ is in $S$. 
\end{lemma}
\begin{proof}
Suppose on the contrary that there exists a graph $G$ such that no optimal locating-dominating set contains every support vertex of $G$. Furthermore, let $S$ be an optimal locating-dominating set of $G$ such that it contains the largest number of support vertices among all optimal locating-dominating sets of $G$. Furthermore, let $s$ be a support vertex of $G$ not in $S$ and let $u_1, u_2, \ldots , u_k$ be all the leaves adjacent to $s$. Since $S$ is a dominating set, we have $u_i \in S$ for all $i \in [k]$. We show that, contrary to the maximality of $S$, the set $S'=(S\setminus\{u_1\})\cup\{s\}$ is a locating-dominating set of $G$. Indeed, the set $S\setminus\{u_1\}$ separates all pairs of vertices in $V(G)\setminus (S\cup\{s,u_1\})$. Furthermore, since $S$ is a dominating set of $G$ and $s \notin S$, any neighbour $v$ other than $u_1$ of $s$ has $(N_G[v] \setminus \{s\}) \cap S \ne \emptyset$. This implies that the vertex $u_1$ is the only vertex in $V(G) \setminus S'$ with $I$-set $\{s\}$ in $S'$. Therefore, set $S'$ is locating-dominating, a contradiction. 
\end{proof}

\begin{lemma}\label{LemLDLeaves}
Let $G$ be a connected graph on at least three vertices without open twins of degree~$1$. Then $G$ admits  an optimal locating-dominating set $S$ such that there are no leaves of $G$ in $S$. 
\end{lemma}
\begin{proof}
Suppose on the contrary that there exists a graph $G$ without open twins of degree~$1$ such that every optimal locating-dominating set contains a leaf of $G$. Consider optimal locating-dominating set $S$ such that it contains the least number of leaves among all optimal locating-dominating sets of $G$ which contains every support vertex of $G$. Notice that $S$ exists by Lemma~\ref{LemLDSup}. Hence, there exist adjacent vertices $s,u\in S$ such that $s$ is a support vertex and $u$ is a leaf. Since $S$ is optimal, the set $S\setminus\{u\}$ is not locating-dominating. Therefore, by the optimality of $S$, there exists a unique vertex $v\not\in S$ such that $I(v)=\{s\}$. Since $G$ has no twins of degree~$1$, the vertex $v$ is not a leaf. However, now the set $S'=(S\cup\{v\})\setminus\{u\}$ is a locating-dominating set of $G$. Since $S'$ is optimal, contains all support vertices and fewer leaves of $G$ than $S$, it contradicts the minimality of $S$. Therefore, the claim follows.
\end{proof}

\begin{lemma}\label{LemC4neighb}
Let $G$ be a subcubic twin-free and triangle-free graph. If vertices $u$ and $v$ are in the same $4$-cycle, then all of their common neighbours are in the same $4$-cycle.
\end{lemma}
\begin{proof}
Assume first that $u$ and $v$ are adjacent. If both of them are adjacent to $w$, then vertices $u,v,w$ form a triangle, a contradiction. Hence, we assume that there is a four-cycle $u,w,v,z$ with edges $uw,wv,vz,zu$. Observe that if a vertex $b\not\in \{u,w,v,z\}$ is adjacent to both $u$ and $v$, then $N(u)=N(v)=\{w,z,b\}$ since we consider a subcubic graph. This is a contradiction with $G$ being twin-free. Hence, the claim follows.
\end{proof}

\begin{lemma}\label{LemTriEdge}
Let $G$ be a subcubic twin-free graph. No two triangles in $G$ share a common edge.
\end{lemma}
\begin{proof}
Let vertices $u,v,w$ form a triangle in $G$. If this triangle shares an edge with another triangle, then without loss of generality two of the vertices, say, $v$ and $w$ have a common neighbour $z$. Hence, we have $N[v]=N[w]=\{u,v,w,z\}$. This is a contradiction with $G$ being a twin-free. Thus, the claim follows.
\end{proof}

With Lemma~\ref{LemTriEdge}, we obtain the following corollary.
\begin{corollary}\label{CorIndC4}
Let $G$ be a subcubic twin-free graph. If vertices $v_1,v_2,v_3,v_4$ form a $4$-cycle, then that $4$-cycle is an induced subgraph of $G$.
\end{corollary}

\section{Subcubic graphs}\label{secMain}

\begin{proposition}\label{propSubcubic}
Let $G$ be a twin-free subcubic graph on $n$ vertices without isolated vertices. We have \[\LD(G)\leq \frac{n}{2}.\]
\end{proposition}
\begin{proof}
Let $G$ be a twin-free subcubic graph on $n$ vertices without isolated vertices. Hence, $n\geq4$. If $n=4$, then $G$ is a path $P_4$ on four vertices. We have $\LD(P_4)=2$. Thus, the claim holds for all subcubic twin-free graphs without isolated vertices on four vertices. Let us next assume on the contrary that $G$ is a graph with the smallest number of vertices and among those graphs one with the smallest number of edges for which the claimed upper bound does not hold. Notice that $G$ is connected. Indeed, if there are multiple components, then the claimed upper bound does not hold on at least one of them and we could have chosen $G$ as that component. However, this contradicts the minimality of $G$.

We first divide the proof based on whether $G$ has triangles. 

\begin{case}[$G$ is triangle-free]

By Theorem~\ref{The4cycle}, we may assume that $G$ contains a $4$-cycle. Denote the vertices in this $4$-cycle by $a,b,c$ and $d$ so that there are edges $ab$, $bc$, $cd$ and $da$. Call this 4-cycle $abcd$. By Corollary~\ref{CorIndC4} there are no edges $ac$ nor $bd$. Furthermore, by Lemma~\ref{LemC4neighb}, the only common neighbours of vertices $a,b,c$ and $d$ are in the set $\{a,b,c,d\}$. Since $G$ is open-twin-free, we may immediately observe that there are at most two vertices of degree~$2$ in a $4$-cycle in $G$. \smallskip

\begin{subcase}[There are exactly two vertices of degree~$2$ in the $4$-cycle $abcd$]

Observe that when there are two vertices of degree~$2$ in the $4$-cycle $abcd$, they are adjacent. Suppose, without loss of generality, that $a$ and $b$ are such vertices. Denote $G_{a,b}=G-a-b$. Notice that $G_{a,b}$ is connected and subcubic. Let us next divide our considerations based on whether $G_{a,b}$ is twin-free.

\begin{subsubcase}[$G_{a,b}$ is twin-free]

Let $S_{a,b}$ be an optimal locating-dominating set in $G_{a,b}$. By the minimality of $G$, we have $|S_{a,b}|\leq \frac{n}{2}-1$. Assume first that $c$ or $d$ is in $S_{a,b}$. We may further assume, without loss of generality that $d\in S_{a,b}$. Consider next $S=S_{a,b}\cup \{a\}$ in $G$. We have $a \in I_G(S,b)$ and $b$ is the only vertex in $V(G)\setminus S$ adjacent to $a$. Thus, $I_G(S,b)$ is unique in $V(G)\setminus S$. Furthermore, every other vertex in $V(G)\setminus S$ is dominated and pairwise separated from other vertices by the same vertices in $S_{a,b}$ as in $G_{a,b}$. Thus, we may assume that neither of $c$ nor $d$ is in $S_{a,b}$. 

Since $S_{a,b}$ is a locating-dominating set, vertices $c$ and $d$ are dominated by some other vertices in  $S_{a,b}$ which are neither adjacent to $a$ nor $b$ in $G$. Hence, we may again consider set $S=S_{a,b}\cup \{a\}$ in $G$. Again, $b$ is the only vertex with $I_G(S,b)=\{a\}$ while all other vertices in $V(G)\setminus S$ are pairwise separated and dominated by the same vertices in $S_{a,b}$ as in $G_{a,b}$. Thus, when a $4$-cycle contains two vertices of degree~$2$ in $G$ and $G_{a,b}$ is twin-free, we have $\LD(G)\leq \frac{n}{2}$. \hfill $\blacktriangleleft \blacktriangleleft$
\end{subsubcase}

\begin{subsubcase}[$G_{a,b}$ contains twins]
 
Notice that since $G$ is twin-free, at least one of the twins is $c$ or $d$. First of all, if $c$ and $d$ are twins, then either $G$ is a cycle on four vertices or it contains a triangle. This contradicts the twin- or triangle-freeness of $G$. Furthermore, if both vertices $c$ and $d$ are twins with $c'$ and $d'$, respectively, then $c'$ and $d'$ have degree~$2$ and $d'$ is adjacent to $c$, and $c'$ is adjacent to $d$. Thus, $G$ contains exactly six vertices and the set $S=\{a,d,c'\}$ is a locating-dominating set of $G$ containing exactly half of the vertices in $G$.

Let us assume next, without loss of generality, that exactly $c$ is a twin with vertex $c'$. Let us denote the other neighbour of $c$ with $e$ and the third neighbour of $e$ by $f$ (see Figure~\ref{fig_case112}). We have $N_G(c')=\{d,e\}$ and $N_{G}(e)=\{c,c',f\}$. Notice that the degree of $e$ is exactly~$3$ since $G$ is subcubic and $e$ is not a twin of $d$. Assume first that $f$ is a leaf. In this case, $G$ contains exactly seven vertices and the set $S=\{b,d,e\}$ is a locating-dominating set containing less than half of the vertices in $G$. Thus, we may assume that $f$ is not a leaf. Consider next the graph $G'=G_{a,b}-c-d=G-a-b-c-d$. We notice that $e$ is a support vertex and $c'$ is a leaf in this graph. Moreover, since $f$ is not a leaf and since $G$ is twin-free, the graph $G'$ is twin-free. By Lemmas~\ref{LemLDLeaves} and~\ref{LemLDSup}, we may assume that $S'$ is an optimal locating-dominating set of $G'$ such that it contains all support vertices but no leaves in $G'$ (in particular, $e \in S'$ and $c' \notin S'$). 
Moreover, by the minimality of $G$, we have $|S'|\leq \frac{n}{2}-2$. Consider next the set $S=S'\cup\{b,d\}$. We have $I(a)=\{d,b\}$, $I(c)=\{b,d,e\}$ and $I(c')=\{d,e\}$. Moreover, all other vertices in $V(G)\setminus S$ are dominated and pairwise separated by the same vertices in $S'$ as in $G$. Hence, $S$ is locating-dominating in $G$ with the claimed cardinality. 
\hfill \mbox{$\blacktriangleleft \blacktriangleleft$}
\end{subsubcase}

Thus, if $G$ contains two vertices of degree~$2$ in the same $4$-cycle, then the result holds.\hfill $\blacktriangleleft$ 
\end{subcase}


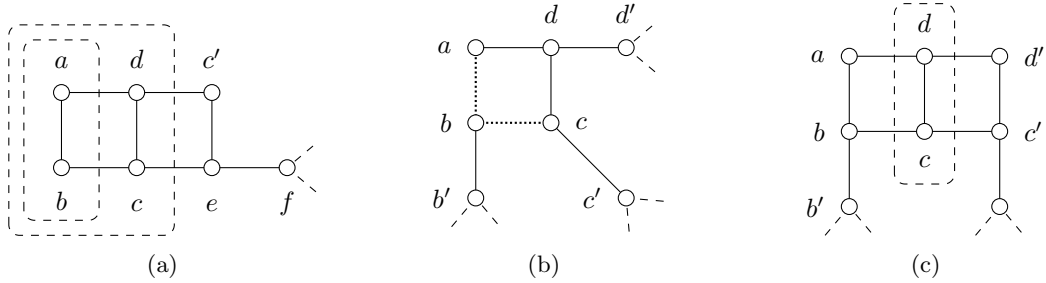
\begin{figure}
	\centering
	\begin{subfigure}[b]{0.32\linewidth}
		\centering
		\begin{tikzpicture}[label distance=2mm]
			\coordinate[label=above:$a$] (a) at (0,1);
			\coordinate[label=below:$b$] (b) at (0,0);
			\coordinate[label={[label distance=3mm]below:$c$}] (c) at (1,0);
			\coordinate[label=above:$d$] (d) at (1,1);
			\coordinate[label=above:$c'$] (c') at (2,1);
			\coordinate[label={[label distance=3mm]below:$e$}] (e) at (2,0);
			\coordinate[label=below:$f$] (f) at (3,0);
			\draw (f) -- (e) -- (c) -- (b) -- (a) -- (d) -- (c') -- (e)   (d) -- (c);
			\draw[dashed] ($(f) + (320:.5)$) -- (f) -- ($(f) + (40:.5)$);
			\draw \foreach \x in {(a),(b),(c),(d),(c'),(e),(f)}{
				\x node[circle, draw, fill=white, inner sep=0pt, minimum width=6pt] {}
			};
			\draw[dashed,rounded corners] ($(b) - (.5,.7)$) rectangle ($(a) + (.5,.7)$) ;
			\draw[dashed,rounded corners] ($(b) - (.7,.9)$) rectangle ($(d) + (.5,.9)$) ;
		\end{tikzpicture}
		\caption{ }\label{fig_case112}
	\end{subfigure}
	\hfill
	\begin{subfigure}[b]{0.32\linewidth}
		\centering
		\begin{tikzpicture}[label distance=2mm]
			\coordinate[label={left:$a$}] (a) at (0,1);
			\coordinate[label=left:$b$] (b) at (0,0);
			\coordinate[label=right:$c$] (c) at (1,0);
			\coordinate[label=above:$d$] (d) at (1,1);
			\coordinate[label=left:$b'$] (b') at (0,-1);
			\coordinate[label=left:$c'$] (c') at (2,-1);
			\coordinate[label=above:$d'$] (d') at (2,1);
			\draw (a) -- (d) -- (c);
			\draw[thick, densely dotted] (a) -- (b) -- (c);
			\draw (b) -- (b')  (c) -- (c')  (d) -- (d');
			\draw[dashed] ($(b') + (230:.5)$) -- (b') -- ($(b') + (310:.5)$);
			\draw[dashed] ($(c') + (355:.5)$) -- (c') -- ($(c') + (275:.5)$);
			\draw[dashed] ($(d') + (320:.5)$) -- (d') -- ($(d') + (40:.5)$);
			\draw \foreach \x in {(a),(b),(c),(d),(b'),(c'),(d')}{
				\x node[circle, draw, fill=white, inner sep=0pt, minimum width=6pt] {}
			};
		\end{tikzpicture}
	\caption{ }\label{fig_case123}
	\end{subfigure}
	\hfill
	\begin{subfigure}[b]{0.32\linewidth}
		\centering
		\begin{tikzpicture}[label distance=2mm]
			\coordinate[label={left:$a$}] (a) at (0,1);
			\coordinate[label=left:$b$] (b) at (0,0);
			\coordinate[label=below:$c$] (c) at (1,0);
			\coordinate[label=above:$d$] (d) at (1,1);
			\coordinate[label=left:$b'$] (b') at (0,-1);
			\coordinate[label=right:$c'$] (c') at (2,0);
			\coordinate[label=right:$d'$] (d') at (2,1);
			\coordinate (e) at (2,-1);
			\draw (a) -- (d) -- (c) -- (b) -- (a);
			\draw (b) -- (b')  (c) -- (c') -- (d') -- (d)  (c') -- (e);
			\draw[dashed] ($(b') + (230:.5)$) -- (b') -- ($(b') + (310:.5)$);
			\draw[dashed] ($(e) + (230:.5)$) -- (e) -- ($(e) + (310:.5)$);
			\draw \foreach \x in {(a),(b),(c),(d),(b'),(c'),(d'),(e)}{
				\x node[circle, draw, fill=white, inner sep=0pt, minimum width=6pt] {}
			};
			\draw[rounded corners,dashed] ($(c) + (-0.4,-0.7)$) rectangle ($(d) + (0.4,0.7)$);
		\end{tikzpicture}
		\caption{ }\label{fig_case1232}
	\end{subfigure}
	\caption{Illustrations for (a) Case 1.1.2, (b) Case 1.2.3 and (c) Case 1.2.3.2. The dotted lines indicate edges being removed. Dashed edges are edges, that may or may not exist. Dashed outlines indicate vertices being removed.}
\end{figure}


From now on we may assume that there is at most one vertex of degree~$2$ in a $4$-cycle in $G$.

\begin{subcase}[There is exactly one vertex of degree~$2$ in the $4$-cycle $abcd$]

Let us say, without loss of generality, that the vertex of degree~$2$ is $a$. Let us denote the third neighbour outside of the set $\{a,b,c,d\}$ of $b$ by $b'$, of $c$ by $c'$ and of $d$ by $d'$ (see Figure~\ref{fig_case123}). 

\begin{subsubcase}[Both $b'$ and $d'$ are leaves]

Let us denote $G'=G-a-d-d'$. Observe that $G'$ is twin-free since $G$ is twin-free, $b$ is a support vertex and $c$ is the only non-leaf adjacent to $b$. Let $S'$ be an optimal locating-dominating set in $G'$ which contains $b$ but does not contain $b'$
(such a set exists by Lemmas~\ref{LemLDSup} and~\ref{LemLDLeaves}). 
By the minimality of $G$, we have $|S'|\leq \frac{n}{2}-1$. Notice that to separate $b'$ and $c$ we have $\{c,c'\}\cap S'\neq \emptyset$. Hence, the set $S=S'\cup\{d\}$ is a locating-dominating set in $G$. Indeed, we have $I(d')=\{d\}$, $I(a)=\{d,b\}$, $I(b')=\{b\}$ and $|I(c)|\geq3$ (if $c \notin S$). Moreover, we have $|S|\leq \frac{n}{2}$.
\hfill $\blacktriangleleft \blacktriangleleft$
\end{subsubcase}

\begin{subsubcase}[Exactly one of $b'$ and $d'$ is a leaf]

Let us assume, without loss of generality, that $b'$ is a leaf while $d'$ is a non-leaf. In this case, we consider the graph $G_{ab,cd}=G-ab-cd$. Notice that this graph is twin-free since $G$ is twin-free, $d'$ is a non-leaf and $c$ is the only non-leaf adjacent to $b$ while $b'$ is the only leaf adjacent to $b$. Furthermore, let $S_{ab,cd}$ be an optimal locating-dominating set in $G_{ab,cd}$ such that it does not contain any leaves and contains all support vertices. It exists by Lemmas~\ref{LemLDSup} and~\ref{LemLDLeaves}. Moreover, by the minimality of $G$ we have $|S_{ab,cd}|\leq \frac{n}{2}$. In particular, we have $b,d\in S_{ab,cd}$. Furthermore, since $I_{G_{ab,cd}}(S_{ab,cd},c)\neq I_{G_{ab,cd}}(S_{ab,cd},b')$, we have $c\in S_{ab,cd}$ or  $c'\in S_{ab,cd}$. Let us next consider set $S_{ab,cd}$ in $G$. Notice that $I_G(S_{ab,cd},a)=\{b,d\}$, $I_G(S_{ab,cd},b')=\{b\}$, $|I_G(S_{ab,cd},c)|\geq3$ (if $c \notin S_{ab,cd}$) and $b$ separates the vertices $d'$ and $c$. Thus, $S_{ab,cd}$ is a locating-dominating set of claimed cardinality in $G$. \hfill \mbox{$\blacktriangleleft \blacktriangleleft$}
\end{subsubcase}

From now on we assume that when a $4$-cycle has exactly one degree~$2$ vertex, neither neighbour of the degree two vertex is a support vertex.

\begin{subsubcase}[Neither $b'$ nor $d'$ is a leaf]

Let us next consider graph $G_{ab,bc}=G-ab-bc$ (see Figure~\ref{fig_case123}). We further divide the proof based on three possibilities: Either $G_{ab,bc}$ is twin-free, or vertex $b$ is a twin with some other vertex or vertex $c$ is a twin with some other vertex. There cannot exist any other twins in $G_{ab,bc}$. Indeed, $G$ is twin-free, $d$ is the support vertex adjacent to $a$ and the leaf $a$ is not a twin since $d'$ is not a leaf. 

\begin{subsubsubcase}[$G_{ab,bc}$ is twin-free]

Let $S_{ab,bc}$ be an optimal locating-dominating set which does not contain any leaves and contains all the support vertices in $G_{ab,bc}$. The set $S_{ab,bc}$ exists by Lemmas~\ref{LemLDLeaves} and~\ref{LemLDSup}. Furthermore, by the minimality of $G$ it has cardinality of at most $\frac{n}{2}$. Observe that $d\in S_{ab,bc}$ and $b'\in S_{ab,bc}$. Observe further that if $c\in S_{ab,bc}$ and $c'\not\in S_{ab,bc}$, then $S_{ab,bc}'=(S_{ab,bc}\setminus\{c\})\cup\{c'\}$ is also a locating-dominating set in $G_{ab,bc}$. Indeed, $c$ is the only vertex with $I(S_{ab,bc}',c)=\{d,c'\}$ since $I(a)=\{d\}$ and $|I(S_{ab,bc},d')|\geq2$ so if $c'\in N(d')$, then $|I(S_{ab,bc}',d')|\geq3$ and $d'$ is separated from all other vertices.

We claim that $S_{ab,bc}$ or $S_{ab,bc}'$ is a locating-dominating set of claimed cardinality in $G$. Let us first consider the set $S_{ab,bc}$. Since $I_G(S_{ab,bc},a)=\{d\}$, the only $I$-set which might change when we consider $G$ instead of $G_{ab,cd}$ is $I(b)$. We have $I_G(S_{ab,bc},b)=\{b'\}$ or $I_G(S_{ab,bc},b)=\{b',c\}$. If $I_G(S_{ab,bc},b)=\{b'\}$, then no $I$-set is modified when we change the perspective from $G_{ab,bc}$ to $G$ and in this case set $S_{ab,bc}$ is a locating-dominating set in $G$. On the other hand, if $I_G(S_{ab,bc},b)=\{b',c\}$, then it is possible that $I_G(S_{ab,bc},b)=I_G(S_{ab,bc},c')$. If this is not the case, then $S_{ab,bc}$ is a locating-dominating set in $G$. However, if $I(b)=I(c')$, then $c'\not\in S_{ab,bc}$ and we may consider the set $S_{ab,bc}'$. Notice that this change does not modify $I(a)$. Moreover, we have $I_G(S_{ab,bc}',b)=\{b'\}$. Since $S_{ab,bc}'$ is a locating-dominating set in $G_{ab,bc}$ and no $I$-sets are modified when we transfer to $G$, the set $S_{ab,bc}'$ is a locating-dominating set also in $G$. Moreover, both $S_{ab,bc}$ and $S_{ab,bc}'$ satisfy the claimed upper bound on the cardinality.
\hfill \mbox{$\blacktriangleleft \blacktriangleleft \blacktriangleleft$}
\end{subsubsubcase}

\begin{subsubsubcase}[$c$ is a twin in $G_{ab,bc}$]

Notice that since $c$ is adjacent to exactly $d$ and $c'$ in $G_{ab,bc}$ and $N(d)=\{a,c,d'\}$, the vertex $c$ is a twin with $d'$. Notice that since $c,c',d',d$ is a $4$-cycle and $d'$ has degree~$2$ in $G$, we have $\deg(c')=3$ and $c'$ is not a support vertex (see Figure~\ref{fig_case1232}). Let us next consider the graph $G_{c,d}=G-c-d$. Since $G$ is twin-free and neither $c'$ nor $b$ is a support vertex in $G$, the graph $G_{c,d}$ is twin-free. Hence, there exists an optimal locating-dominating set $S_{c,d}$ containing no leaves and every support vertex in $G_{c,d}$ by Lemmas~\ref{LemLDLeaves} and~\ref{LemLDSup} with cardinality $|S_{c,d}|\leq \frac{n}{2}-1$. Notice that $b,c'\in S_{c,d}$. Let us consider the set $S=S_{c,d}\cup\{d\}$. We have $I_G(S,a)=\{b,d\}$,  $I_G(S,c)=\{b,d,c'\}$ and $I_G(S,d')=\{c',d\}$. All other $I$-sets remain unmodified and since $S_{c,d}$ is a locating-dominating set in $G_{c,d}$, the set $S$ is locating-dominating in $G$.
\hfill $\blacktriangleleft \blacktriangleleft \blacktriangleleft$
\smallskip
\end{subsubsubcase}

\begin{subsubsubcase}[$b$ is a twin in $G_{ab,bc}$]\label{case_btwin}

Notice that since $b$ is a leaf in $G_{ab,bc}$, there is a leaf adjacent to $b'$ in $G$. Let us call this leaf $b''$. Consider the graph $G_{bb'}=G-bb'$ (see Figure~\ref{fig_case1233}). Observe that either $G_{bb'}$ is twin-free or vertices $b'$ and $b''$ form a two-vertex component $P_2$. Note that in both cases, due to minimality of $G$ and since $\LD(P_2)=1$, there exists a locating-dominating set $S_{bb'}$ in $G_{bb'}$ which has cardinality at most $\frac{n}{2}$ and contains all support vertices and no leaves in $G_{bb'}$ by Lemmas~\ref{LemLDLeaves} and~\ref{LemLDSup} (if $b'$ and $b''$ form a $P_2$-component we consider $b'$ as a support vertex and $b''$ as a leaf).  In particular, we have $b'\in S_{bb'}$. Hence, when we consider $S_{bb'}$ in $G$, the only $I$-set which changes between $G$ and $G_{bb'}$ is $I(b)$. Thus, we only need to confirm that $I(b)$ is unique when $b\not\in S_{bb'}$. First of all, observe that $|I_G(b)|\geq2$. If $a\in I_G(b)$, then $I_G(b)$ is unique since the edge $db'$ does not exist by Lemma~\ref{LemC4neighb}. Thus, we may assume that $c\in I_G(b)$ and $a\not\in I_G(b)$. Hence, if $I_G(b)=I_G(x)$, then $x\in N(c)$ and $x=d$ or $x=c'$. However, by Lemma~\ref{LemC4neighb}, we have $d\not\in N(b')$. Thus, $x=c'$. 

Hence, we need to only consider the case $I_G(b)=I_G(c')=\{c,b'\}$. In particular, we may assume that the edge $c'b'$ exists in $G$.
In this case, we consider the graph $G_{b',b''}=G-b'-b''$. This graph is twin-free since $\deg(d)=\deg(c)=3$ and hence, neither $b$ nor $c'$ may become a twin with a removal of their neighbour. Due to the minimality of $G$, there exists an optimal locating-dominating set $S_{b',b''}$ in $G_{b',b''}$ with cardinality $|S_{b',b''}|\leq \frac{n}{2}-1$. Furthermore, the set $S=S_{b',b''}\cup\{b''\}$ is a locating-dominating set of cardinality at most $\frac{n}{2}$ in $G$. Indeed, each $I$-set of a vertex in $V(G_{b',b''})$ remains unmodified while $b''\in I(S, b')$ and no other vertex is adjacent to $b''$. 
Therefore, the claim follows also in this case.
\hfill $\blacktriangleleft \blacktriangleleft \blacktriangleleft$ 
\end{subsubsubcase}
\hfill $\blacktriangleleft \blacktriangleleft$ 
\end{subsubcase}
\hfill $\blacktriangleleft$ 
\end{subcase}


\begin{figure}
	\centering
	\begin{subfigure}[b]{0.30\linewidth}
		\centering
		\begin{tikzpicture}[label distance=1.5mm,scale=0.9]
			\coordinate[label={left:$a$}] (a) at (0,1);
			\coordinate[label=above:$b$] (b) at (1,1);
			\coordinate[label=right:$c$] (c) at (1,0);
			\coordinate[label=left:$d$] (d) at (0,0);
			\coordinate[label=above:$b'$] (b') at (2,1);
			\coordinate[label=left:$c'$] (c') at (2,-1);
			\coordinate[label=below:$b''$] (b'') at (3,1);
			\draw (a) -- (d) -- (c);
			\draw (a) -- (b) -- (c);
			\draw[thick, densely dotted] (b) -- (b');
			\draw (b') -- (b'')  (c) -- (c');
			\draw[dashed] (b') -- ($(b') + (270:.5)$);
			\draw ($(d) + (270:.5)$) -- (d);
			\draw[dashed] ($(c') + (355:.5)$) -- (c') -- ($(c') + (275:.5)$);
			\draw \foreach \x in {(a),(b),(c),(d),(b'),(c'),(b'')}{
				\x node[circle, draw, fill=white, inner sep=0pt, minimum width=6pt] {}
			};
		\end{tikzpicture}
		\caption{ }\label{fig_case1233}
	\end{subfigure}
	\hfill
	\begin{subfigure}[b]{0.34\linewidth}
		\centering
		\begin{tikzpicture}[label distance=2mm]
			\coordinate[label={above:$a$}] (a) at (0,1);
			\coordinate[label=above:$b$] (b) at (1,1);
			\coordinate[label={[label distance=3mm]below:$c$}] (c) at (1,0);
			\coordinate[label=below:$d$] (d) at (0,0);
			\coordinate[label=above:$a'$] (a') at (-1,1);
			\coordinate[label=above:$b'$] (b') at (2,1);
			\coordinate[label=below:$c'$] (c') at (2,0);
			\coordinate[label=below:$d'$] (d') at (-1,0);
			\coordinate[label=above:$a''$] (a'') at (-2,1);
			\draw[thick, densely dotted] (a) -- (b) -- (c) -- (d) -- (a);
			\draw (a) -- (a') -- (d') -- (d)  (b) -- (b') -- (c') -- (c)  (a') -- (a'');
			\draw ($(d') + (180:.5)$) -- (d');
			\draw ($(b') + (0:.5)$) -- (b');
			\draw ($(c') + (0:.5)$) -- (c');
			\draw \foreach \x in {(a),(b),(c),(d),(a'),(b'),(c'),(d'),(a'')}{
				\x node[circle, draw, fill=white, inner sep=0pt, minimum width=6pt] {}
			};
			\draw[rounded corners, dashed] ($(a'') - (.5,.3)$) rectangle ($(a') + (.5,.7)$); 
		\end{tikzpicture}
		\caption{ }\label{fig_case132}
	\end{subfigure}
	\hfill
	\begin{subfigure}[b]{0.32\linewidth}
	\centering
	\begin{tikzpicture}[label distance=1.5mm,scale=0.9]
		\coordinate[label={above:$a$}] (a) at (0,1);
		\coordinate[label=above:$b$] (b) at (1,1);
		\coordinate[label={right:$c$}] (c) at (1,0);
		\coordinate[label=left:$d$] (d) at (0,0);
		\coordinate[label=above:$a'$] (a') at (-1,1);
		\coordinate[label=above:$b'$] (b') at (2,1);
		\coordinate[label=right:$c'$] (c') at (1,-1);
		\coordinate[label=above:$d'$] (d') at (-1,-1);
		\coordinate[label=right:$b''$] (b'') at (2,0);
		\draw (a) -- (d) -- (c);
		\draw[thick, densely dotted] (a) -- (b) -- (c);
			\draw (a) -- (a')  (b) -- (b')  (c) -- (c')  (d') -- (d);
		\draw[dashed] ($(a') + (140:.5)$) -- (a') -- ($(a') + (220:.5)$);
		\draw[dashed] ($(b') + (0:.5)$) -- (b') -- (b'') ;
		\draw[dashed] ($(c') + (230:.5)$) -- (c') -- ($(c') + (310:.5)$);
		\draw[dashed] ($(d') + (195:.5)$) -- (d') -- ($(d') + (265:.5)$);
		\draw \foreach \x in {(a),(b),(c),(d),(a'),(b'),(c'),(d'),(b'')}{
			\x node[circle, draw, fill=white, inner sep=0pt, minimum width=6pt] {}
		};
	\end{tikzpicture}
	\caption{ }\label{fig_case133}
	\end{subfigure}
	\caption{Illustrations for (a) Case 1.2.3.3, (b) Case 1.3.2 and (c) Case 1.3.3. The dotted edges indicate edges being removed. Dashed edges are edges, that may or may not exist. Dashed outlines indicate vertices being removed.}
\end{figure}
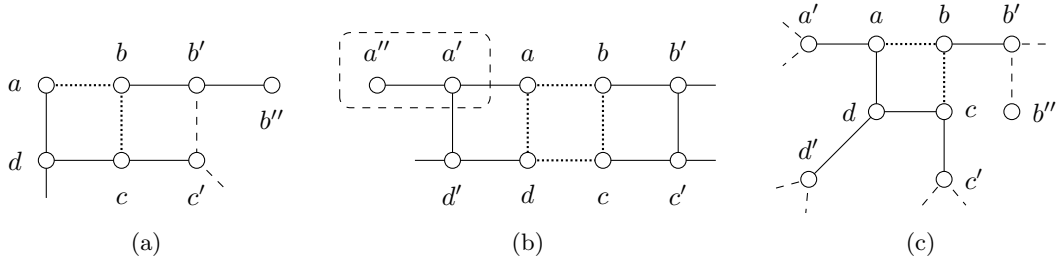


Therefore, we may assume from now on that there are no $4$-cycles with vertices of degree~$2$.

\begin{subcase}[Every vertex in a $4$-cycle has degree $3$]

We denote the neighbours of $a$, $b$, $c$ and $d$ that are outside of the $4$-cycle by $a'$, $b'$, $c'$ and $d'$, respectively. Recall that due to Lemma~\ref{LemC4neighb} the vertices $a'$, $b'$, $c'$ and $d'$ are distinct. We divide the proof into cases based on which of the possible edges between $a'$, $b'$, $c'$ and $d'$ are present.

\begin{subsubcase}[The edges $a'b'$, $b'c'$, $c'd'$ and $d'a'$ are present in $G$]
 
The entire graph $G$ is now determined. Indeed, we have $G= P_2 \Box C_4$ and the conjectured bound holds, since $G$ is a cubic graph (see~\cite{foucaud2016location}).
\hfill $\blacktriangleleft \blacktriangleleft$
\end{subsubcase}

Therefore, we may assume that at least one of the edges in the subcase above is not present. Without loss of generality, we assume that the edge $a'b'$ is not present. The proof is then divided into cases based on whether the incident edges $b'c'$ and $a'd'$ are present in $G$.

\begin{subsubcase}[The edge $a'b'$ is not present but the edges $b'c'$ and $a'd'$ are present in~$G$]

If none of the vertices $a'$, $b'$, $c'$ and $d'$ are support vertices, then the graph $G' = G - ab - bc - cd - da$ is clearly twin-free (see Figure~\ref{fig_case132}, note that the figure is drawn in a simplified way and we may have some additional edges such as $a'c'$). The vertices $a$, $b$, $c$ and $d$ are leaves, and the vertices $a'$, $b'$, $c'$ and $d'$ are support vertices in $G'$. By the minimality of the number of edges of $G$, there exists a locating-dominating set $S'$ of $G'$ such that $|S'| \leq \frac{n}{2}$. Due to Lemmas~\ref{LemLDSup} and~\ref{LemLDLeaves}, we may assume that $a',b',c',d' \in S'$ and $a,b,c,d \notin S'$. The $I$-sets given by $S'$ are identical in $G'$ and $G$, and thus $S'$ is a locating-dominating set of $G$ with $|S'| \leq \frac{n}{2}$.

Assume then that at least one of $a'$, $b'$, $c'$ and $d'$, say $a'$, is a support vertex. (Notice that if $c'$ or $d'$ is a support vertex, then the edge $c'd'$ is not present, and these cases are symmetrical to $a'$ being a support vertex.) Let $a''$ be the leaf attached to $a'$ (see Figure~\ref{fig_case132}). Consider the graph $G_{a',a''} = G - a' - a''$. The only vertices that might have twins in $G_{a',a''}$ are $a$ and $d'$. The vertex $a$ does not have a twin, since $N_{G_{a',a''}} (a) = \{b,d\}$ and the only other vertex adjacent to both $b$ and $d$ is $c$, but $c'$ is a neighbour of $c$ that is not adjacent to $a$. If $d'$ has a twin, then that twin must be $a$, $c$ or $d$. The vertex $a$ has no twins, $b$ is adjacent to $c$ but not $d'$, and $c$ is adjacent to $d$ but not $d'$. Thus, the vertex $d'$ has no twins either. Therefore, the graph $G_{a',a''}$ is twin-free. By the minimality of $G$, there exists a locating-dominating set $S_{a',a''}$ of $G_{a',a''}$ with cardinality at most $\frac{n}{2} - 1$. Now the set $S = S_a \cup \{a'\}$ is a locating-dominating set of $G$ since $a''$ is the only vertex with an $I$-set containing only $a'$. Since $|S| \leq \frac{n}{2}$, the claim holds.
\hfill $\blacktriangleleft \blacktriangleleft$
\end{subsubcase}

\begin{subsubcase}[The edges $a'b'$ and $b'c'$ are not present in $G$]

Consider the graph $G_{ab,bc} = G - ab - bc$. The vertex $b$ is a leaf, and $b'$ is a support vertex in $G_{ab,bc}$ (see Figure~\ref{fig_case133}). 

\begin{subsubsubcase}[$G_{ab,bc}$ is twin-free]

There exists a locating-dominating set $S_{ab,bc}$ of $G_{ab,bc}$ such that $b' \in S_{ab,bc}$ and $b \notin S_{ab,bc}$ (by Lemmas~\ref{LemLDSup} and~\ref{LemLDLeaves}), and $|S_{ab,bc}| \leq \frac{n}{2}$. Since $b \notin S_{ab,bc}$, $I(b)$ is the only $I$-set given by $S_{ab,bc}$ that can be different in $G$ when compared to $G_{ab,bc}$. Indeed, if $a,c \notin S_{ab,bc}$, then all $I$-sets in $G$ are identical to the $I$-sets in $G_{ab,bc}$.
If $a \in S_{ab,bc}$ or $c \in S_{ab,bc}$, then $I_G(b)$ contains $a$ or $c$, but the rest of the $I$-sets are identical to those of $G_{ab,bc}$.
The only vertices whose $I$-sets could be the same as $I_G(b)$ are $a'$, $c'$ and $d$, but $b' \in I_G (b)$ and $b'$ is not adjacent to $a'$, $c'$ or $d$ (due to the edges $a'b'$ and $b'c'$ not being present and Lemma~\ref{LemC4neighb}). Thus, $I_G(b)$ is unique and $S_{ab,bc}$ is a locating-dominating set of $G$ with $|S_{ab,bc}| \leq  \frac{n}{2}$.
\hfill $\blacktriangleleft \blacktriangleleft \blacktriangleleft$
\end{subsubsubcase}

\begin{subsubsubcase}[$G_{ab,bc}$ is not twin-free]
 
Now at least one of $a$, $b$ and $c$ has a twin. Suppose that $a$ has a twin. Since $N_{G_{ab,bc}} (a) = \{a',d\}$, $d'$ and $c$ are the only possible twins of $a$. If $d'$ is a twin with $a$, then $a a' d' d a$ is a cycle in $G$ and the degree of $d'$ is two in $G$. This contradicts our assumption that all vertices in a $4$-cycle have degree~$3$ in $G$. If $c$ is a twin with $a$, then $a' = c'$, but this is impossible due to Lemma~\ref{LemC4neighb}. Thus, neither $a$ nor $c$ (by symmetry) have twins in $G_{ab,bc}$. Therefore, $b$ has a twin in $G_{ab,bc}$. Now either $b$ and $b'$ form a $P_2$-component in $G_{ab,bc}$ or $b'$ has a leaf $b''$ in $G$. 
These cases are handled somewhat similarly as Case~\ref{case_btwin}. 

Assume that $b$ and $b'$ form a $P_2$-component in $G_{ab,bc}$. The graph $G_{b,b'} = G - b - b'$ is twin-free, and thus there exists a locating-dominating set $S_{b,b'}$ such that $|S_{b,b'}| \leq \frac{n}{2} - 1$. The set $S = S_{b,b'} \cup \{b\}$ is clearly a locating-dominating set of $G$, and we have $|S| \leq \frac{n}{2}$.

Assume then that $b'$ has a leaf $b''$ in $G$. Consider the graph $G_{bb'} = G - bb'$. Again, $G_{bb'}$ is either twin-free or $b'$ and $b''$ form a $P_2$-component in $G_{bb'}$. As in the previous case, if $b'$ and $b''$ form a $P_2$-component in $G_{bb'}$, we can easily construct a locating-dominating set $S$ of $G$ such that $|S| \leq \frac{n}{2}$ by considering a locating-dominating set of $G - b' - b''$. So assume that $G_{bb'}$ is twin-free. There exists a locating-dominating set $S_{bb'}$ of $G_{bb'}$ such that $|S_{bb'}| \leq \frac{n}{2}$ and $b' \in S_{bb'}$. We claim that the set $S_{bb'}$ is also a locating-dominating set of $G$. The only $I$-set that might differ between the two graphs is $I(b)$ (assuming $b \notin S_{bb'}$). Since $S_{bb'}$ is a locating-dominating set of $G_{bb'}$, we have $a \in S_{bb'}$ or $c \in S_{bb'}$. Now, if $I_G(b)$ is the same as the $I$-set of some other vertex, then that vertex must be $a'$, $d$ or $c'$. However, we also have $b' \in I_G(b)$ and $b'$ is not adjacent to $a'$, $d$ or $c'$. Thus, $I_G(b)$ is unique, and the set $S_{bb'}$ is a locating-dominating set of $G$ with $|S_{bb'}| \leq \frac{n}{2}$.
\hfill $\blacktriangleleft \blacktriangleleft \blacktriangleleft$
\end{subsubsubcase}
\hfill $\blacktriangleleft \blacktriangleleft$
\end{subsubcase}
\hfill $\blacktriangleleft$
\end{subcase}
\end{case}

Therefore, $\LD(G) \leq \frac{n}{2}$ holds for all triangle-free twin-free subcubic graphs with no isolated vertices. We then assume that $G$ is not triangle-free.

\begin{case}[$G$ has triangles as induced subgraphs]

Let us assume $T = G[a,b,c]$ to be a triangle in $G$ induced by the vertices $a$, $b$ and $c$. If any two vertices of $T$ are of degree~2 in $G$, then it implies that the said vertices are twins in $G$, a contradiction to our assumptions. Hence, we assume from here on that at most one vertex of $T$ is of degree~2 in $G$.

\begin{subcase}[$T$ has a vertex of degree~2 in $G$]

Without loss of generality, let us assume that $\deg_G(a)=2$. This implies that the vertices $b$ and $c$ have neighbours, say $b'$ and $c'$, respectively, in $G$ outside of $T$ (see Figure~\ref{fig_case21}). Observe that $b' \neq c'$, or else, the pair $b$ and $c$ would be twins in $G$, a contradiction to our assumption. Let $G_{ab} = G - ab$.

\begin{subsubcase}[$G_{ab}$ is twin-free]

By our assumption on the minimality of the graph $G$, there exists an LD-set $S_{ab}$ of $G_{ab}$ such that $|S_{ab}| \leq \frac{n}{2}$. Moreover, by Lemmas~\ref{LemLDSup} and~\ref{LemLDLeaves}, since $c$ is a support vertex and $a$ is a leaf in $G_{ab}$, we can assume that $c \in S_{ab}$ and $a \notin S_{ab}$. We then show that the set $S_{ab}$ is also an LD-set of $G$. If $b \notin S_{ab}$, then $I_G(x) = I_{G_{ab}}(x)$ for each $x\in V(G)\setminus S_{ab}$, and thus $S_{ab}$ is an LD-set of~$G$.

Let us, therefore, assume next that $b \in S_{ab}$. Now, if $S_{ab}$ is not an LD-set of $G$, it would imply that there exists a vertex $x$ of $G$ other than $a$ and not in $S_{ab}$ such that the pair $a,x$ is separated in $G_{ab}$ but not in $G$. Since $I_G(a) = \{b,c\}$, we must have $I_G(x) = \{b,c\}$ which makes the vertices $b$ and $c$ twins in $G$, a contradiction to our assumption. Hence, $S_{ab}$ is an LD-set of $G$ also in the case that $a \notin S_{ab}$ and $b \in S_{ab}$.

Thus, overall, $S_{ab}$ is an LD-set of $G$ with $|S_{ab}| \leq \frac{n}{2}$ and thus, the result follows in the case that the graph $G_{ab}$ is twin-free. \hfill $\blacktriangleleft \blacktriangleleft$
\end{subsubcase}
Now, by symmetry, we may assume that, if the graph $G_{ac} = G - ac$ is also twin-free, then the result holds as well. 


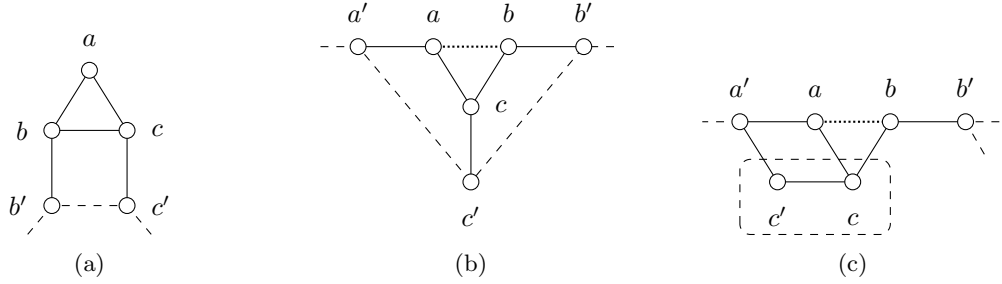
\begin{figure}
	\centering
	\begin{subfigure}[b]{0.32\linewidth}
		\centering
		\begin{tikzpicture}[label distance=2mm]
			\coordinate[label={above:$a$}] (a) at (.5,.8);
			\coordinate[label=left:$b$] (b) at (0,0);
			\coordinate[label={right:$c$}] (c) at (1,0);
			\coordinate[label={left:$b'$}] (b') at (0,-1);
			\coordinate[label={right:$c'$}] (c') at (1,-1);
			\draw (a) -- (b) -- (c) -- (a)  (b) -- (b')  (c) -- (c');
			\draw[dashed] ($(b') + (230:.5)$) -- (b') -- (c') -- ($(c') + (310:.5)$);
			\draw \foreach \x in {(a),(b),(c),(b'),(c')}{
				\x node[circle, draw, fill=white, inner sep=0pt, minimum width=6pt] {}
			};
		\end{tikzpicture}
		\caption{ }\label{fig_case21}
	\end{subfigure}
	\hfill
	\begin{subfigure}[b]{0.32\linewidth}
		\centering
		\begin{tikzpicture}[label distance=2mm]
			\coordinate[label={above:$a$}] (a) at (0,1);
			\coordinate[label=above:$b$] (b) at (1,1);
			\coordinate[label=right:$c$] (c) at (.5,.2);
			\coordinate[label=above:$a'$] (a') at (-1,1);
			\coordinate[label=above:$b'$] (b') at (2,1);
			\coordinate[label=below:$c'$] (c') at (.5,-.8);
			\draw (b) -- (c) -- (a)  (b) -- (b')  (a) -- (a')  (c) -- (c');
			\draw[thick,densely dotted] (a) -- (b);
			\draw[dashed] ($(a') + (180:.5)$) -- (a') -- (c') -- (b') -- ($(b') + (0:.5)$);
			\draw \foreach \x in {(a),(b),(c),(a'),(b'),(c')}{
				\x node[circle, draw, fill=white, inner sep=0pt, minimum width=6pt] {}
			};
		\end{tikzpicture}
		\caption{ }\label{fig_case22}
	\end{subfigure}
	\hfill
	\begin{subfigure}[b]{0.32\linewidth}
		\centering
		\begin{tikzpicture}[label distance=2mm]
			\coordinate[label={above:$a$}] (a) at (0,1);
			\coordinate[label=above:$b$] (b) at (1,1);
			\coordinate[label={[label distance=3mm]below:$c$}] (c) at (.5,.2);
			\coordinate[label=above:$a'$] (a') at (-1,1);
			\coordinate[label=above:$b'$] (b') at (2,1);
			\coordinate[label=below:$c'$] (c') at (-.5,.2);
			\draw (b) -- (c) -- (a)  (b) -- (b')  (a) -- (a') -- (c') -- (c);
			\draw[thick,densely dotted] (a) -- (b);
			\draw[dashed] (a') -- ($(a') + (180:.5)$);
			\draw[dashed] ($(b') + (300:.5)$) -- (b') -- ($(b') + (0:.5)$);
			\draw \foreach \x in {(a),(b),(c),(a'),(b'),(c')}{
				\x node[circle, draw, fill=white, inner sep=0pt, minimum width=6pt] {}
			};
			\draw[rounded corners,dashed] ($(c') - (.5,.7)$) rectangle ($(c) + (.5,.3)$) ;
		\end{tikzpicture}
		\caption{ }\label{fig_case222}
	\end{subfigure}
	\caption{Illustrations for (a) Case 2.1, (b) Case 2.2 and (c) Case 2.2.2. The dotted edges indicate edges being removed. Dashed edges are edges, that may or may not exist. Dashed outlines indicate vertices being removed.}
\end{figure}


\begin{subsubcase}[Both $G_{ab}$ and $G_{ac}$ have twins]\label{Case2ab+actwins}

Let us first look at the graph $G_{ab}$. The twins in $G_{ab}$ must either be a pair $a,x$ or a pair $b,y$, where $x$ and $y$ are vertices of $G$ different from $a$ and $b$, respectively.

\begin{subsubsubcase}[$a,x$ are twins in $G_{ab}$ for some $x \in V(G) \setminus \{a\}$]

In this case, since $ac \in E(G)$, we must have $x \in N_G(c) \setminus \{a\} = \{b,c'\}$. If $x = b$, that is, if $a$ and $b$ are twins in $G_{ab}$, it implies a contradiction since $\deg_{G_{ab}}(a) = 1 \ne 2 = \deg_{G_{ab}}(b)$. Therefore, $x \ne b$. In other words, we have $x = c'$, that is, $a$ and $c'$ are twins in $G_{ab}$. Therefore, we must also have $\deg_G(c') = \deg_{G_{ab}}(c') = \deg_{G_{ab}}(a) = 1$. We now look at the graph $G_{ac} = G - ac$. By our assumption on Case~\ref{Case2ab+actwins}, the graph $G_{ac}$ also has twins. However, $G_{ac}$ cannot have twins of the form $c,z$ for some vertex $z\neq c$ of $G$, since the neighbour $c'$ of $c$ is of degree~$1$ in $G_{ac}$. Therefore, by analogy to the previous case when $a$ and $c'$ were twins in $G_{ab}$, the vertices $a$ and $b'$ must be twins in $G_{ac}$. Again, by the same analogy, we infer that $\deg_G(b') = 1$. Therefore, the graph $G$ is determined and it can be checked that the set $S = \{b,c\}$ is an LD-set of $G$ with $|S| < \frac{1}{2} \times 5 = \frac{n}{2}$. Hence, the result follows. \hfill $\blacktriangleleft \blacktriangleleft \blacktriangleleft$
\end{subsubsubcase}
Again, by symmetry, we may assume that the result also holds in the case where $a$ is a twin in graph $G_{ac}$. Thus, we assume from here on that in graphs $G_{ab}$ or $G_{ac}$ the vertex $a$ does not belong to a pair of twins.

\begin{subsubsubcase}[$b,y$ are twins in $G_{ab}$ for some vertex $y \in V(G) \setminus \{a,b\}$]

In this case, since $bc \in E(G)$, we must have $y \in N_G(c) \setminus \{a,b\} = \{c'\}$. Therefore, we must have $y=c'$ with $b'c' \in E(G)$. Moreover, $\deg_{G}(c') = \deg_{G_{ab}}(c') = \deg_{G_{ab}}(b) = 2$. We again look at the graph $G_{ac}$ which, by 
assumption on Case~\ref{Case2ab+actwins}, has twins 
other than the pair $a,b'$. Therefore, by symmetry to the graph $G_{ab}$,  $b'$ and $c$ must be twins in $G_{ac}$ with $\deg_G(b') = 2$. Therefore, the graph $G$ is again determined and it can be checked that the set $S = \{b,c\}$ is an LD-set of $G$ with $|S| = 2 < \frac{1}{2} \times 5 = \frac{n}{2}$. Hence, the result holds in this case. \hfill $\blacktriangleleft \blacktriangleleft \blacktriangleleft$
\end{subsubsubcase}
Hence, the result follows in the case that both $G_{ab}$ and $G_{ac}$ have twins. \hfill $\blacktriangleleft \blacktriangleleft$
\end{subsubcase}
In conclusion, the claim holds when $T$ has a vertex of degree~$2$ in $G$. \hfill $\blacktriangleleft$ 
\end{subcase}
\begin{subcase}[Each vertex of $T$ is of degree~$3$ in $G$]

By assumption, we have $\deg_G(a) = \deg_G(b) = \deg_G(c) = 3$. Let $N_G(a) \setminus \{b,c\} = a'$, $N_G(b) \setminus \{a,c\} = b'$ and $N_G(c) \setminus \{a,b\} = c'$. Notice that each of $a'$, $b'$ and $c'$ must be distinct, or else, $G$ would have twins, a contradiction to our assumption. If $a'b', a'c', b'c' \in E(G)$, then the graph is a cubic graph in which the vertex subset $S = \{a,b,c\}$ can be checked to be an LD-set of order $|S| = 3 = \frac{1}{2} \times 6 = \frac{n}{2}$. Hence, in this case, the result holds. Let us, therefore, assume without loss of generality that $a'b' \notin E(G)$ (see Figure~\ref{fig_case22}). We now consider the graph $G_{ab} = G - ab$.

\begin{subsubcase}[$G_{ab}$ is twin-free]

By the minimality of $G$, let us assume that $S_{ab}$ is an LD-set of $G_{ab}$ such that $|S_{ab}| \leq \frac{n}{2}$. We show that the set $S_{ab}$ is also an LD-set of $G$. Now, if either $a,b \in S_{ab}$ or $a,b \notin S_{ab}$, then the set $S_{ab}$ is also an LD-set of $G$ and we are done. Indeed, in these cases we have $I_G(S_{ab} ; x) = I_{G_{ab}}(S_{ab} ; x)$ for each $x\not\in S_{ab}$. Hence, by symmetry and therefore without loss of generality, let us assume that $a \in S_{ab}$ and $b \notin S_{ab}$. Let us 
next suppose on the contrary that $S_{ab}$ is not an LD-set of $G$. Then, the only way that can happen is if there exists a vertex $y$ of $G$ other than $b$ and not in $S_{ab}$ such that the pair $b,y$ is separated in $G_{ab}$ but not in $G$. In particular, we must have $y \in N_{G_{ab}}(a) = \{a',c\}$. 

Let us first assume that $y = a'$, that is, the pair $a',b$ is not separated by $S_{ab}$ in $G$. We must have $|\{b',c\} \cap S_{ab}| \ge 1$ in order for $S_{ab}$ to dominate $b$. Now, if $b' \in S_{ab}$, it implies that $a'b' \in E(G)$ contrary to our assumption. Therefore, $b' \notin S_{ab}$. This implies that $c \in S_{ab}$. This further implies that $a'c \in E(G)$, thus making the pair $a,c$ twins in $G$, a contradiction to our assumption. Hence, we conclude that $y \ne a'$. Let us therefore assume now that $y = c$, that is the pair $b,c$ is not separated by $S_{ab}$ in $G$. Therefore, in particular, we have $c \notin S_{ab}$. Therefore, in order for $S_{ab}$ to dominate $b$, we must have $b' \in S_{ab}$ which implies that $b'c \in E(G)$. This further implies that the vertices $b$ and $c$ are twins in $G$, a contradiction to our assumption. Therefore, $y \ne c$ either. In other words, this proves that $S_{ab}$ is indeed an LD-set of $G$.

Hence, in the case that the graph $G_{ab}$ is twin-free, we find an LD-set $S_{ab}$ of $G$ such that $|S_{ab}| \le \frac{n}{2}$ and thus, the result holds. \hfill $\blacktriangleleft \blacktriangleleft$
\end{subsubcase}

\begin{subsubcase}[$G_{ab}$ has twins]

In this case, a twin pair in $G_{ab}$ is either of the form $a,x$ or of the form $b,y$, where $x$ and $y$ are vertices of $G$ different from $a$ and $b$, respectively. By symmetry, let us consider a pair $a,x$ to be twins in $G_{ab}$. Now, since $ac \in E(G)$, we must have $x \in N_G(c) \setminus \{a\} = \{b,c'\}$. If $x = b$, that is, if $a$ and $b$ are twins in $G_{ab}$, then they are twins in $G$ as well, a contradiction to our assumption. Therefore, $x \ne b$ and so, we have $x = c'$, that is, $a$ and $c'$ are twins in $G_{ab}$. This implies that $a'c' \in E(G)$ (see Figure~\ref{fig_case222}). Moreover, $\deg_G(c') = \deg_{G_{ab}}(c') = \deg_{G_{ab}}(a) = 2$. We next consider the graph $G_{c,c'} = G - c - c'$.

Observe first that $G_{c,c'}$ is connected since $N_G(c')=\{a',c\}$ and $N_G(c)=\{a,b,c'\}$. Furthermore, graph $G_{c,c'}$ is also twin-free. Indeed, if $x\in V(G_{c,c'})$ is a twin with some vertex $y\in V(G_{c,c'})$, then we may assume, without loss of generality, that $x\in\{a,b,a'\}$. We have $N_{G_{c,c'}}(a)=\{b,a'\}$. Thus, if $x=a$, then $y=b'$ and the edge $a'b'\in E(G)$, a contradiction. By symmetry, we have that $x\neq b$ and $y\not\in \{a,b\}$ and hence, $x=a'$. However, now $y\in N_{G_{c,c'}}(a)$. Thus, $y=b$, a contradiction. Therefore, graph $G_{c,c'}$ is twin-free and, by the minimality of $G$, it admits an LD-set $S_{c,c'}$ of size at most $\frac{n}{2}-1$.

Observe that for set $S_{c,c'}$ to dominate $a$ in $G_{c,c'}$, we have $\{a,b,a'\}\cap S_{c,c'}\neq \emptyset$. First assume that $a'\in S_{c,c'}$.
\begin{subsubsubcase}[$a'\in S_{c,c'}$]
We consider set $S=S_{c,c'}\cup\{c'\}$. In particular, set $S$ is dominating in $G$ and it separates all vertex pairs $x,y\in V(G_{c,c'})\setminus S_{c,c'}$ using the vertices in $S_{c,c'}$. Furthermore, vertex $c$ is the only vertex in $V(G)\setminus S_{c,c'}$ adjacent to $c'\in S$. Thus, $S$ separates all vertices in $V(G)\setminus S$ and it is an LD-set in graph $G$ with cardinality at most $\frac{n}{2}$. \hfill $\blacktriangleleft \blacktriangleleft \blacktriangleleft$
\end{subsubsubcase}

Assume next that $a$ or $b$ is in $S_{c,c'}$.
\begin{subsubsubcase}[$a\in S_{c,c'}$ or $b\in S_{c,c'}$]
We consider set $S=S_{c,c'}\cup\{c\}$. Clearly, $S$ is dominating in $G$. Furthermore, as in the previous case, set $S$ separates all vertices $x,y\in V(G_{c,c'})\setminus S_{c,c'}$ using the vertices in $S_{c,c'}$. Finally, vertex $c'$ is adjacent to $c$ but neither of the vertices $a$ nor $b$. Thus, also $c'$ is separated from all other vertices in $V(G)\setminus S$. Hence, $S$ separates all vertices in $V(G)\setminus S$ and it is an LD-set in graph $G$ with cardinality at most $\frac{n}{2}$. \hfill $\blacktriangleleft \blacktriangleleft \blacktriangleleft$
\end{subsubsubcase}

Therefore, graph $G$ admits an LD-set of claimed cardinality in the case that $G_{ab}$ has twins.\hfill \flushright{$\blacktriangleleft \blacktriangleleft$}
\end{subsubcase}

Thus, in the case that all three vertices of $T$ are of degree~$3$ in $G$, the result follows. \hfill $\blacktriangleleft$
\end{subcase}

Therefore, the theorem holds if the graph $G$ has triangles as induced subgraphs.
\end{case}

Finally, Cases 1 and 2 together prove the theorem.
\end{proof}


\begin{theorem}\label{theclosedtwins}
Let $G$ be a connected open-twin-free subcubic graph on $n$ vertices without isolated vertices other than $K_3$ or $K_4$. We have \[\LD(G)\leq \frac{n}{2}.\]
\end{theorem}
\begin{proof}
Let us assume on the contrary, that there exists an open-twin-free subcubic graph $G$ on $n$ vertices without isolated vertices other than $K_3$ or $K_4$ for which we have $\LD(G)> \frac{n}{2}.$ Furthermore, let us assume that $G$ has the fewest number of closed twins among these graphs. By Proposition~\ref{propSubcubic}, there is at least one pair of closed twins in $G$. Notice that since $G$ is not $K_4$, we cannot have a triple of pairwise closed twins. Moreover, notice that closed twins have degree 2 or 3.

Let us first assume that there exist closed twins $u,v$ of degree $2$ in $G$ and that they are adjacent to vertex $w$. Notice that now $\deg(w)=3$. Consider graph $G'=G-vw$. In this graph, vertex $v$ is a leaf, there are no open twins and the number of closed twins is smaller than in $G$. Hence, $\LD(G')\leq \frac{n}{2}$. Let us denote by $S'$ an optimal locating-dominating set in $G'$. By Lemma~\ref{LemLDSup}, we may assume that $u\in S'$ and by Lemma~\ref{LemLDLeaves}, we may assume that $v\not\in S'$. Moreover, since $I_{G'}(w)\neq I_{G'}(v)=\{u\}$, we have some vertex $z\in I_{G'}(w)$. However, now $S'$ is a locating-dominating set in $G$. Indeed, we either have $z\neq w$ and $I_G(v)=\{u\}$ while $I_G(w)=\{z,u\}$. If on the other hand $z=w$, then $v$ is separated from all other vertices in $V(G)\setminus S'$ since $v$ is the only vertex in $V(G)\setminus S'$ which is adjacent to $u$.

Let us next assume that $u$ and $v$ are closed twins of degree 3 in $G$. Let us denote $N[u]=N[v]=\{u,v,w,w'\}$. Notice that if there is an edge between $w$ and $w'$, then $G$ is the complete graph $K_4$. Moreover, if they are adjacent to the same vertex $z\neq u,v$, then $w$ and $w'$ are open twins. The same is true if $w$ and $w'$ have degree 2. Hence, we may assume that $z\in N(w)\setminus N(w')$. Assume next that $\deg(w')=2$. First notice that if $z$ is a leaf, then $G$ has five vertices and set $\{w,v\}$ is locating-dominating in $G$. Hence, we may assume that $z$ is not a leaf. Consider graph $G_{v,w'}=G-v-w'$. Since $G_{v,w'}$ has fewer closed twins than $G$, there exists a locating-dominating set $S_{v,w'}$ of size at most $\frac{n}{2}-1$ which has $w\in S_{v,w'}$ and $u\not\in S_{v,w'}$ by Lemmas~\ref{LemLDSup} and~\ref{LemLDLeaves}. Furthermore, now set $S=S_{v,w'}\cup\{v\}$ is locating-dominating in $G$. Indeed, $w$ separates vertices $u$ and $w'$ while $v$ separates $w$ from all other vertices of $G$. Hence, $\deg(w')=3$.

Assume next that $z\in N(w)\setminus N(w')$ and $z'\in N(w')\setminus N(w)$. Consider graph $G'=G-uw-uw'-vw$. In this graph $u$ is a leaf and $v$ is its support vertex. Now also $w$ and $z$ are a leaf and a support vertex, respectively, or $w$ and $z$ form a $P_2$-component. Moreover, there is a possibility that we created a pair of open twins if $z$ is a support vertex in $G$. 

Assume first that $w$ and $z$ form a $P_2$-component in $G'$. Now $G_{w,z}' = G'-w-z$ contains fewer closed twins than $G$. Thus, there exists an optimal locating-dominating set $S'$ of $G_{w,z}'$ that contains $v$ and $|S'|\leq \frac{n}{2} - 1$. Now, $S = S' \cup \{z\}$ is a locating-dominating set of $G'$ that contains both $v$ and $z$ such that $|S|\leq \frac{n}{2}$. Moreover, the vertex $w'$ is dominated by at least two vertices. When we consider the set $S$ in $G$, the only possible vertices in $V(G)\setminus S$ which might not be separated are $u$, $w$ and $w'$. However, $w$ is only of these vertices adjacent to $z$ and $w'$ is dominated by at least two vertices. Hence, either $w'\in S$ or $z'\in S$. In both cases, set $S$ is locating-dominating in $G$.

Assume then that $z$ is a support vertex in $G'$ but not in $G$. The number of closed twins in $G'$ is smaller than in $G$. Thus, we have an optimal locating-dominating set $S$ in $G'$ which contains vertices $z$ and $v$ such that $|S|\leq \frac{n}{2}$. By the same arguments as in the case above concerning the $P_2$-component, $S$ is locating-dominating also in $G$.

Let us next assume that $z$ is a support vertex and $\ell_z$ is the adjacent leaf in $G$. In this case, we consider subgraph $G''=G-uw-vw'$. Notice that $G''$ does not contain any open twins and it has a smaller number of closed twins than $G$. Hence, it admits an optimal locating-dominating set $S'$ with $|S'|\leq \frac{n}{2}$. By Lemma~\ref{LemLDSup} we may assume that $z\in S'$ and by Lemma~\ref{LemLDLeaves} that $\ell_z\not\in S'$. Hence, we have $w\in S'$ or $v\in S'$ for separating $\ell_z$ and $w$. If $w\in S'$, then also set $S_w=(S'\setminus \{w\})\cup\{v\}$ is a locating-dominating set in $G''$. Moreover, $S_w$ is a locating-dominating set in $G$. Indeed, the only vertices in $G$ which might not be separated are $w,w'$ and $u$. However, $w$ is the only one of these vertices adjacent to $z$. If $w'$ and $u$ are not separated in $G$, then $w', u \notin S_w$. However, $w'$ is dominated by $S_w$ in $G''$, and thus $z' \in S_w$ and it separates $w'$ and $u$ also in $G$. Therefore, $S_w$ is a locating-dominating set of cardinality $|S_w|\leq \frac{n}{2}$ in $G$.  Now, the claim follows.
\end{proof}


\section{Subcubic graphs with open twins of degree~$3$}\label{SecOpentwins}

Let $G$ be any graph, $F$ be a connected graph and $U$ be a vertex subset of $F$. Then a subgraph $H$ of $G$ is called an \emph{$(F;U)$-subgraph} of $G$ if the following hold.
\begin{enumerate}
\item There exists an isomorphism $j : V(F) \to V(H)$.

\item $N_G[j(u)] \subseteq V(H)$ and $j(u)j(v)\in E(H)$ for all $u \in U$ and $j(v)\in N_G(j(u))$.
\end{enumerate}

We note that an $(F;U)$-subgraph $H$ of $G$ could also be considered as a subgraph of $G$ isomorphic to $F$ such that the closed neighbourhood in $G$ of any vertex $j(u)$ for $u\in U$ together with the edges in $G$ incident with $j(u)$ are also contained in the subgraph $H$.


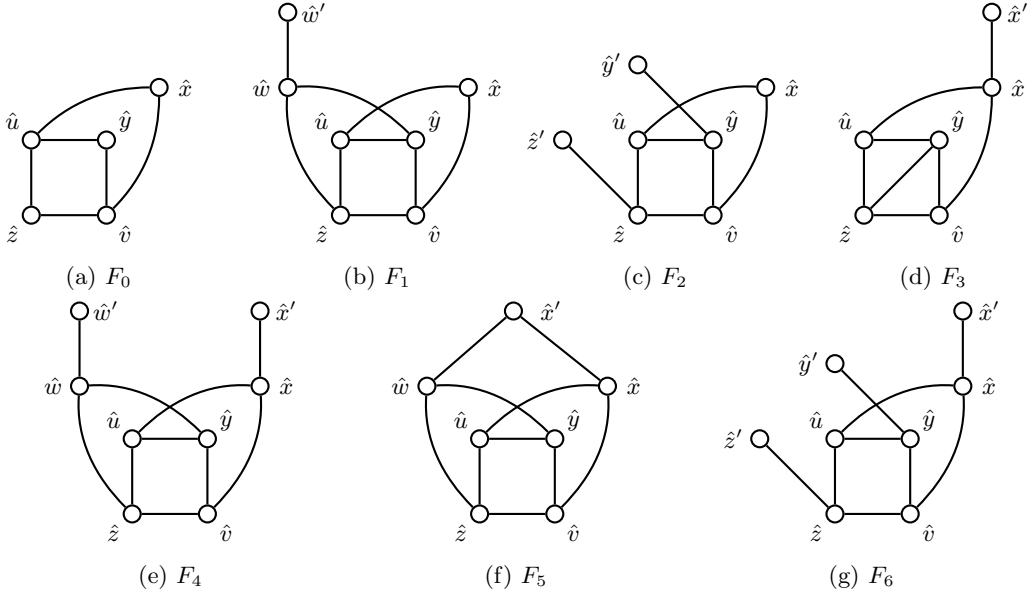
\begin{figure}[t!]
\centering
\begin{subfigure}[t]{0.24\textwidth}
\centering
\begin{tikzpicture}[
blacknode/.style={circle, draw=black!, fill=black!, thick},
whitenode/.style={circle, draw=black!, fill=white!, thick},
scale=0.5]
\tiny
\node[whitenode] (z) at (0,0) {}; \node at (-0.5,-0.5) {\small $\hat{z}$};
\node[whitenode] (v) at (2,0) {}; \node at (2.5,-0.5) {\small $\hat{v}$};
\node[whitenode] (y) at (2,2) {}; \node at (2.5,2.5) {\small $\hat{y}$};
\node[whitenode] (u) at (0,2) {}; \node at (-0.5,2.5) {\small $\hat{u}$};
\node[whitenode] (x) at (3.4,3.4) {}; \node at (4.1,3.4) {\small $\hat{x}$};

\draw[-, thick] (u) -- (y);
\draw[-, thick] (y) -- (v);
\draw[-, thick] (v) -- (z);
\draw[-, thick] (z) -- (u);

\draw[-, thick] (u) ..controls (1,3) and (2,3.4).. (x);
\draw[-, thick] (v) ..controls (3,1) and (3.4,2).. (x);

\end{tikzpicture}
\caption{$F_0$}\label{Fig_F0}
\end{subfigure}
\begin{subfigure}[t]{0.24\textwidth}
\centering
\begin{tikzpicture}[
blacknode/.style={circle, draw=black!, fill=black!, thick},
whitenode/.style={circle, draw=black!, fill=white!, thick},
scale=0.5]
\tiny
\node[whitenode] (z) at (0,0) {}; \node at ($(z)+(-0.5,-0.5)$) {\small $\hat{z}$};
\node[whitenode] (v) at (2,0) {}; \node at ($(v)+(0.5,-0.5)$) {\small $\hat{v}$};
\node[whitenode] (y) at (2,2) {}; \node at ($(y)+(0.5,0.5)$) {\small $\hat{y}$};
\node[whitenode] (u) at (0,2) {}; \node at ($(u)+(-0.5,0.5)$) {\small $\hat{u}$};
\node[whitenode] (x) at ($(y)+(1.4,1.4)$) {}; \node at ($(x)+(0.7,0)$) {\small $\hat{x}$};
\node[whitenode] (w) at ($(u)+(-1.4,1.4)$) {}; \node at ($(w)+(-0.7,0)$) {\small $\hat{w}$};
\node[whitenode] (w') at ($(w)+(0,2)$) {}; \node at ($(w')+(0.7,0)$) {\small $\hat{w}'$};

\draw[-, thick] (u) -- (y);
\draw[-, thick] (y) -- (v);
\draw[-, thick] (v) -- (z);
\draw[-, thick] (z) -- (u);

\draw[-, thick] (u) ..controls ($(u)+(1,1)$) and ($(y)+(0,1.5)$).. (x);
\draw[-, thick] (v) ..controls ($(v)+(1,1)$) and ($(y)+(1.5,0)$).. (x);
\draw[-, thick] (y) ..controls ($(y)+(-1,1)$) and ($(u)+(0,1.5)$).. (w);
\draw[-, thick] (z) ..controls ($(z)+(-1,1)$) and ($(u)+(-1.5,0)$).. (w);

\draw[-, thick] (w) -- (w');

\end{tikzpicture}
\caption{$F_1$}\label{Fig_F1}
\end{subfigure}
\begin{subfigure}[t]{0.24\textwidth}
\centering
\begin{tikzpicture}[
blacknode/.style={circle, draw=black!, fill=black!, thick},
whitenode/.style={circle, draw=black!, fill=white!, thick},
scale=0.5]
\tiny
\node[whitenode] (z) at (0,0) {}; \node at ($(z)+(-0.5,-0.5)$) {\small $\hat{z}$};
\node[whitenode] (v) at (2,0) {}; \node at ($(v)+(0.5,-0.5)$) {\small $\hat{v}$};
\node[whitenode] (y) at (2,2) {}; \node at ($(y)+(0.5,0.5)$) {\small $\hat{y}$};
\node[whitenode] (u) at (0,2) {}; \node at ($(u)+(-0.5,0.5)$) {\small $\hat{u}$};
\node[whitenode] (x) at ($(y)+(1.4,1.4)$) {}; \node at ($(x)+(0.7,0)$) {\small $\hat{x}$};
\node[whitenode] (y') at ($(y)+(-2,2)$) {}; \node at ($(y')+(-0.7,0)$) {\small $\hat{y}'$};
\node[whitenode] (z') at ($(z)+(-2,2)$) {}; \node at ($(z')+(-0.7,0)$) {\small $\hat{z}'$};

\draw[-, thick] (u) -- (y);
\draw[-, thick] (y) -- (v);
\draw[-, thick] (v) -- (z);
\draw[-, thick] (z) -- (u);

\draw[-, thick] (u) ..controls ($(u)+(1,1)$) and ($(y)+(0,1.5)$).. (x);
\draw[-, thick] (v) ..controls ($(v)+(1,1)$) and ($(y)+(1.5,0)$).. (x);
\draw[-, thick] (y) -- (y');
\draw[-, thick] (z) -- (z');

\end{tikzpicture}
\caption{$F_2$}\label{Fig_F2}
\end{subfigure}
\begin{subfigure}[t]{0.24\textwidth}
\centering
\begin{tikzpicture}[
blacknode/.style={circle, draw=black!, fill=black!, thick},
whitenode/.style={circle, draw=black!, fill=white!, thick},
scale=0.5]
\tiny
\node[whitenode] (z) at (0,0) {}; \node at (-0.5,-0.5) {\small $\hat{z}$};
\node[whitenode] (v) at (2,0) {}; \node at (2.5,-0.5) {\small $\hat{v}$};
\node[whitenode] (y) at (2,2) {}; \node at (2.5,2.5) {\small $\hat{y}$};
\node[whitenode] (u) at (0,2) {}; \node at (-0.5,2.5) {\small $\hat{u}$};
\node[whitenode] (x) at (3.4,3.4) {}; \node at (4.1,3.4) {\small $\hat{x}$};
\node[whitenode] (x') at ($(x)+(0,2)$) {}; \node at ($(x')+(0.7,0)$) {\small $\hat{x}'$};

\draw[-, thick] (u) -- (y);
\draw[-, thick] (y) -- (v);
\draw[-, thick] (v) -- (z);
\draw[-, thick] (z) -- (u);
\draw[-, thick] (y) -- (z);
\draw[-, thick] (x) -- (x');

\draw[-, thick] (u) ..controls (1,3) and (2,3.4).. (x);
\draw[-, thick] (v) ..controls (3,1) and (3.4,2).. (x);

\end{tikzpicture}
\caption{$F_3$}\label{Fig_F3}
\end{subfigure}
\begin{subfigure}[t]{0.3\textwidth}
\centering
\begin{tikzpicture}[
blacknode/.style={circle, draw=black!, fill=black!, thick},
whitenode/.style={circle, draw=black!, fill=white!, thick},
scale=0.5]
\tiny
\node[whitenode] (z) at (0,0) {}; \node at ($(z)+(-0.5,-0.5)$) {\small $\hat{z}$};
\node[whitenode] (v) at (2,0) {}; \node at ($(v)+(0.5,-0.5)$) {\small $\hat{v}$};
\node[whitenode] (y) at (2,2) {}; \node at ($(y)+(0.5,0.5)$) {\small $\hat{y}$};
\node[whitenode] (u) at (0,2) {}; \node at ($(u)+(-0.5,0.5)$) {\small $\hat{u}$};
\node[whitenode] (x) at ($(y)+(1.4,1.4)$) {}; \node at ($(x)+(0.7,0)$) {\small $\hat{x}$};
\node[whitenode] (x') at ($(x)+(0,2)$) {}; \node at ($(x')+(0.7,0)$) {\small $\hat{x}'$};
\node[whitenode] (w) at ($(u)+(-1.4,1.4)$) {}; \node at ($(w)+(-0.7,0)$) {\small $\hat{w}$};
\node[whitenode] (w') at ($(w)+(0,2)$) {}; \node at ($(w')+(0.7,0)$) {\small $\hat{w}'$};

\draw[-, thick] (u) -- (y);
\draw[-, thick] (y) -- (v);
\draw[-, thick] (v) -- (z);
\draw[-, thick] (z) -- (u);

\draw[-, thick] (u) ..controls ($(u)+(1,1)$) and ($(y)+(0,1.5)$).. (x);
\draw[-, thick] (v) ..controls ($(v)+(1,1)$) and ($(y)+(1.5,0)$).. (x);
\draw[-, thick] (y) ..controls ($(y)+(-1,1)$) and ($(u)+(0,1.5)$).. (w);
\draw[-, thick] (z) ..controls ($(z)+(-1,1)$) and ($(u)+(-1.5,0)$).. (w);

\draw[-, thick] (w) -- (w');
\draw[-, thick] (x) -- (x');

\end{tikzpicture}
\caption{$F_4$}\label{Fig_F4}
\end{subfigure}
\begin{subfigure}[t]{0.3\textwidth}
\centering
\begin{tikzpicture}[
blacknode/.style={circle, draw=black!, fill=black!, thick},
whitenode/.style={circle, draw=black!, fill=white!, thick},
scale=0.5]
\tiny
\node[whitenode] (z) at (0,0) {}; \node at ($(z)+(-0.5,-0.5)$) {\small $\hat{z}$};
\node[whitenode] (v) at (2,0) {}; \node at ($(v)+(0.5,-0.5)$) {\small $\hat{v}$};
\node[whitenode] (y) at (2,2) {}; \node at ($(y)+(0.5,0.5)$) {\small $\hat{y}$};
\node[whitenode] (u) at (0,2) {}; \node at ($(u)+(-0.5,0.5)$) {\small $\hat{u}$};
\node[whitenode] (x) at ($(y)+(1.4,1.4)$) {}; \node at ($(x)+(0.7,0)$) {\small $\hat{x}$};
\node[whitenode] (x') at ($(x)+(-2.5,2)$) {}; \node at ($(x')+(1,0)$) {\small $\hat{x}'$};
\node[whitenode] (w) at ($(u)+(-1.4,1.4)$) {}; \node at ($(w)+(-0.7,0)$) {\small $\hat{w}$};

\draw[-, thick] (u) -- (y);
\draw[-, thick] (y) -- (v);
\draw[-, thick] (v) -- (z);
\draw[-, thick] (z) -- (u);

\draw[-, thick] (u) ..controls ($(u)+(1,1)$) and ($(y)+(0,1.5)$).. (x);
\draw[-, thick] (v) ..controls ($(v)+(1,1)$) and ($(y)+(1.5,0)$).. (x);
\draw[-, thick] (y) ..controls ($(y)+(-1,1)$) and ($(u)+(0,1.5)$).. (w);
\draw[-, thick] (z) ..controls ($(z)+(-1,1)$) and ($(u)+(-1.5,0)$).. (w);

\draw[-, thick] (w) -- (x');
\draw[-, thick] (x) -- (x');

\end{tikzpicture}
\caption{$F_5$}\label{Fig_F5}
\end{subfigure}
\begin{subfigure}[t]{0.3\textwidth}
\centering
\begin{tikzpicture}[
blacknode/.style={circle, draw=black!, fill=black!, thick},
whitenode/.style={circle, draw=black!, fill=white!, thick},
scale=0.5]
\tiny
\node[whitenode] (z) at (0,0) {}; \node at ($(z)+(-0.5,-0.5)$) {\small $\hat{z}$};
\node[whitenode] (v) at (2,0) {}; \node at ($(v)+(0.5,-0.5)$) {\small $\hat{v}$};
\node[whitenode] (y) at (2,2) {}; \node at ($(y)+(0.5,0.5)$) {\small $\hat{y}$};
\node[whitenode] (u) at (0,2) {}; \node at ($(u)+(-0.5,0.5)$) {\small $\hat{u}$};
\node[whitenode] (x) at ($(y)+(1.4,1.4)$) {}; \node at ($(x)+(0.7,0)$) {\small $\hat{x}$};
\node[whitenode] (x') at ($(x)+(0,2)$) {}; \node at ($(x')+(0.7,0)$) {\small $\hat{x}'$};
\node[whitenode] (y') at ($(y)+(-2,2)$) {}; \node at ($(y')+(-0.7,0)$) {\small $\hat{y}'$};
\node[whitenode] (z') at ($(z)+(-2,2)$) {}; \node at ($(z')+(-0.7,0)$) {\small $\hat{z}'$};

\draw[-, thick] (u) -- (y);
\draw[-, thick] (y) -- (v);
\draw[-, thick] (v) -- (z);
\draw[-, thick] (z) -- (u);

\draw[-, thick] (u) ..controls ($(u)+(1,1)$) and ($(y)+(0,1.5)$).. (x);
\draw[-, thick] (v) ..controls ($(v)+(1,1)$) and ($(y)+(1.5,0)$).. (x);
\draw[-, thick] (y) -- (y');
\draw[-, thick] (z) -- (z');
\draw[-, thick] (x) -- (x');

\end{tikzpicture}
\caption{$F_6$}\label{Fig_F6}
\end{subfigure}
\caption{A list of pairwise non-isomorphic subcubic graphs. The vertices $\hat{u}$ and $\hat{v}$ are open twins of degree~$3$ in each graph.} \label{Fig_F}
\end{figure}


If $U = \{u_1,u_2, \ldots , u_k\}$ and $H$ is an $(F;U)$-subgraph of $G$, then we may also refer to $H$ as an $(F;u_1,u_2, \ldots , u_k)$-subgraph of $G$. We now define a list of pairwise non-isomorphic graphs as follows. In the proof of Theorem~\ref{thm_deg3}, we show that a connected subcubic graph $G$ with open twins of degree 3 and at least 7 edges has $F_0$ and at least one of graphs $F_i$, $i\in [1,6]$, as its subgraph. 
\begin{itemize}
\item \emph{Graph $F_0$}: $V(F_0) = \{\hat{u},\hat{v},\hat{x},\hat{y},\hat{z}\}$ and $E(F_0) = \{\hat{u} \hat{x}, \hat{u} \hat{y}, \hat{u} \hat{z}, \hat{v} \hat{x}, \hat{v} \hat{y}, \hat{v} \hat{z}\}$. See Figure~\ref{Fig_F0}. The vertices $\hat{u}$ and $\hat{v}$ are open twins of degree~$3$ in $F_0$. Any  graph that has a pair of open twins of degree~$3$ has an $(F_0;\hat{u}, \hat{v})$-subgraph.

\item \emph{Graph $F_1$}: $V(F_1) = \{\hat{u},\hat{v},\hat{x}, \hat{y}, \hat{z}, \hat{w}, \hat{w}'\}$ and $E(F_1) = \{\hat{u} \hat{x}, \hat{u} \hat{y}, \hat{u} \hat{z}, \hat{v} \hat{x}, \hat{v} \hat{y}, \hat{v} \hat{z}, \hat{y} \hat{w}, \hat{z} \hat{w}, \hat{w} \hat{w}'\}$. See Figure~\ref{Fig_F1}. The pairs $\hat{u}, \hat{v}$ and $\hat{y}, \hat{z}$ are open twins of degree~$3$ in $F_1$.

\item \emph{Graph $F_2$}: $V(F_2) = \{\hat{u},\hat{v},\hat{x},\hat{x}',\hat{y},\hat{y}',\hat{z}\}$ and $E(F_2) = \{\hat{u} \hat{x}, \hat{u} \hat{y}, \hat{u} \hat{z}, \hat{v} \hat{x}, \hat{v} \hat{y}, \hat{v} \hat{z}, \hat{y} \hat{y}', \hat{z} \hat{z}'\}$.  See Figure~\ref{Fig_F2}.

\item \emph{Graph $F_3$}: $V(F_3) = \{\hat{u},\hat{v},\hat{x},\hat{x}',\hat{y},\hat{z}\}$ and $E(F_3) = \{\hat{u} \hat{x}, \hat{u} \hat{y}, \hat{u} \hat{z}, \hat{v} \hat{x}, \hat{v} \hat{y}, \hat{v} \hat{z}, \hat{y} \hat{z}, \hat{x} \hat{x}'\}$. See Figure~\ref{Fig_F3}. The pair $\hat{u}, \hat{v}$ is open twins of degree~$3$ in $F_3$ and the pair $\hat{y}, \hat{z}$ is closed twins of degree~$3$ in $F_3$.

\item \emph{Graph $F_4$}: $V(F_4) = \{\hat{u},\hat{v}, \hat{w}, \hat{w}',\hat{x}, \hat{x}', \hat{y}, \hat{z}\}$ and $E(F_4) = \{\hat{u} \hat{x}, \hat{u} \hat{y}, \hat{u} \hat{z}, \hat{v} \hat{x}, \hat{v} \hat{y}, \hat{v} \hat{z}, \hat{w} \hat{y}, \hat{w} \hat{z}, \hat{w} \hat{w}',\\ \hat{x} \hat{x}'\}$. See Figure~\ref{Fig_F4}. 

\item \emph{Graph $F_5$}: $V(F_5) = \{\hat{u},\hat{v}, \hat{w}, \hat{x}, \hat{x}', \hat{y}, \hat{z}\}$ and $E(F_5) = \{\hat{u} \hat{x}, \hat{u} \hat{y}, \hat{u} \hat{z}, \hat{v} \hat{x}, \hat{v} \hat{y}, \hat{v} \hat{z}, \hat{w} \hat{y}, \hat{w} \hat{z}, \hat{x} \hat{x}', \hat{w} \hat{x}'\}$. See Figure~\ref{Fig_F5}. 

\item \emph{Graph $F_6$}: $V(F_6) = \{\hat{u},\hat{v}, \hat{w}, \hat{x}, \hat{x}', \hat{y}, \hat{y}', \hat{z}, \hat{z}'\}$ and $E(F_6) = \{\hat{u} \hat{x}, \hat{u} \hat{y}, \hat{u} \hat{z}, \hat{v} \hat{x}, \hat{v} \hat{y}, \hat{v} \hat{z}, \hat{x} \hat{x}', \hat{y} \hat{y}',\\ \hat{z} \hat{z}'\}$. See Figure~\ref{Fig_F6}. 
\end{itemize}

Let $G$ be a graph, $i \in [0,6]$ and let $U_i$ be a vertex subset of $F_i$. Moreover, let $H_i$ be an $(F_i;U_i)$-subgraph of $G$ under a homomorphism $j_i : V(F_i) \to V(H_i)$. Then we fix a certain naming convention for the vertices of $H_i$ and $F_i$ as follows. Firstly, we fix a set of $10$ symbols, namely $L = \{u,v,w,w',x,x',y,y',z,z'\}$. Then, as shown in Figure~\ref{Fig_F}, any vertex of $F_i$ will be denoted by the symbol $\hat{a} \in L$ for some $a \in L$. In addition, any vertex $j_i(\hat{a})$ of $H_i$, for some $a \in L$, will be denoted by the symbol $a$ (that is, by dropping the hat on the symbol $\hat{a}$). We shall call this the \emph{drop-hat naming convention} on $V(H_i)$. We hope that the conventions and the namings will become clearer to the reader as we proceed further with their usages.


Let $\calG$ denote the set of all connected subcubic graphs without open twins of degrees~$1$ or $2$ and not isomorphic to $K_{3,3}$. We will prove in Theorem~\ref{thm_deg3} that $\LD(G)\leq\frac{n}{2}$ when $G \in \calG$ and $G$ has at least~$7$ edges. In the proof, we consider a minimal counterexample and prove that such a graph cannot exist. To that end, we prove a series of lemmas to show that the counterexample contains no $F_i$, $i\in [1,6]$ as its subgraph.

\begin{lemma}\label{lem-F1}
    If $G\in\calG$ has the smallest possible number of edges $m\geq7$ such that $\LD(G)>\frac{n}{2}$, then $G$ has no $(F_1;\hat{u},\hat{v},\hat{w},\hat{x},\hat{y},\hat{z})$-subgraphs.
\end{lemma}
\begin{proof} 
On the contrary, let us suppose that $G$ contains an $(F_1;\hat{u},\hat{v},\hat{w},\hat{x},\hat{y},\hat{z})$-subgraph, say $H_1$. Under an injective homomorphism $j_1 : V(F_1) \to V(G)$, applying the  drop-hat naming convention, the vertices $j_1(\hat{u})$, $j_1(\hat{v})$, $j_1(\hat{w})$, $j_1(\hat{x})$, $j_1(\hat{y})$ and $j_1(\hat{z})$ of $H_1$ are called $u$, $v$, $w$, $x$, $y$ and $z$. Now, the vertex $x$ is of degree~$2$ in $G$ and the pairs $u,v$ and $y,z$ are open twins of degree~$3$ in $G$. Let $D = \{ux\}$ and let $G' = G - D$. This implies that $G' \in \calG$. Moreover,  we have $8 \le m(G') < m(G)$ and hence, by the minimality of $G$, there exists an LD-set $S'$ of $G'$ such that $|S'| \le \frac{n}{2}$. We further notice that the vertex $x$ is a leaf with its support vertex $v$ in $G'$. Therefore, by Lemmas~\ref{LemLDSup} and~\ref{LemLDLeaves}, we assume that $v \in S'$ and $x \notin S'$. Now, if $u \notin S'$, then $S'$ is also an LD-set of $G$. 

Let us, therefore, assume that $u \in S'$. If, on the contrary, $S'$ is not an LD-set of $G$, it would mean that, in $G$, the vertex $x$ is not separated by $S'$ from some other vertex $p \in N_G(u)\cap N_G(v)\setminus\{S'\cup x'\} \subseteq \{y,z\}$. Now, since $y$ and $z$ are open twins in $G$ (hence, also in $G'$), it implies that $S' \cap \{y,z\} \ne \emptyset$. Without loss of generality, therefore, let us assume that $y \in S'$ and that the vertices $x,z$ are not separated by $S'$ in $G$. In this case, we claim that the set $S = (S' \setminus \{u\}) \cup \{w\}$ is an LD-set of $G$. The set $S$ is a dominating set of $G$ since each vertex in $N_G[u]$ has a neighbour in $\{v,y\} \subset S$. We therefore show that $S$ is also a separating set of $G$. In particular, $y$ separates $u$ from other vertices in $V(G)\setminus S$. While $v$ separates $x$ and $z$ from other vertices in $V(G)\setminus S$ and $w$ separates $x$ and $z$. Since $S'$ is an LD-set of $G'$, set $S$ is locating-dominating in  $G$. Moreover, $|S| = |S'| \le \frac{n}{2}$ contradicts our assumption that $\LD(G) > \frac{n}{2}$. Hence, this proves that $G$ cannot contain any $(F_1;\hat{u},\hat{v},\hat{w},\hat{x},\hat{y},\hat{z})$-subgraph.
\end{proof}

\begin{lemma}\label{lem-F2}
	If $G\in\calG$ has the smallest possible number of edges $m\geq7$ such that $\LD(G)>\frac{n}{2}$, then $G$ has no $(F_2;\hat{u},\hat{v},\hat{x},\hat{y},\hat{z})$-subgraphs.
\end{lemma}
\begin{proof} 
On the contrary,  suppose that $G$ contains an $(F_2;\hat{u},\hat{v},\hat{x},\hat{y},\hat{z})$-subgraph, say $H_2$. Applying the drop-hat naming convention on $V(H_2)$, the vertex $x$ has degree~$2$ in $G$ and the vertices $u,v$ are open twins of degree~$3$ in $G$. Now, if $\deg_G(y') = \deg_G(z') = 1$, then the graph $G$ is determined  to be isomorphic to $F_2$ on $n=7$ vertices. It can be checked in this case that the set $S = \{v,y,z\}$ is an LD-set of $G$ such that $|S| = 3 < \frac{n}{2}$. This contradicts our assumption that $\LD(G) > \frac{n}{2}$. Hence, we assume that at least one of $y'$ and $z'$ has degree of at least $2$ in $G$. In other words, $m \ge 9$. Let $D = \{ux\}$ and let $G' = G - D$. We have $G' \in \calG$. Moreover, $8 \le m(G') < m$ and hence, by the minimality of $G$, there exists an LD-set $S'$ of $G'$ such that $|S'| \le \frac{n}{2}$. We further notice that the vertex $x$ is a leaf with its support vertex $v$ in $G'$. Therefore, by Lemmas~\ref{LemLDSup} and~\ref{LemLDLeaves}, we assume that $v \in S'$ and $x \notin S'$. Now, if $u \notin S'$, then $S'$ is also an LD-set of $G$. 

Let us assume that $u \in S'$. Now, if $S'$ is not an LD-set of $G$, then, in $G$, the vertex $x$ is not separated by $S'$ from some other vertex $p \in N_G(u)\cap N_G(v) \setminus (S' \cup \{x\})\subseteq\{y,z\}$. Now, in order for $S'$ to separate the pair $y,z$ in $G'$, we must have  $\{y,y'\} \cap S' \ne \emptyset$ or $\{z,z'\} \cap S' \ne \emptyset$. Therefore, without loss of generality, let us assume that $\{y,y'\} \cap S' \ne \emptyset$. This implies that $p = z$, that is, the vertices $x,z$ are not separated by $S'$ in $G$. In particular, therefore, we have $z \notin S'$. We now claim that the set $S = (S' \setminus \{u\}) \cup \{z\}$ is an LD-set of $G$. The set $S$ is a dominating set of $G$ since each vertex in $N_G[u]$ has a neighbour in $\{v,z\} \subset S$. Therefore, we show that $S$ is also a separating set of $G$. First of all, $I_G(S;u)=\{z\}$ is unique since $v\in S$ and $z'$ is dominated by some vertex in $S'$. Furthermore, $v$ separates $x$ from all other vertices except $y$. However, we had $\{y,y'\}\cap S'\neq \emptyset$. Thus, $y\in S$ or $y$ and $x$ are separated by $S$. Since $S'$ is locating-dominating in $G'$, also all other vertices in $V(G)\setminus S$ are pairwise separated. Therefore, $S$ is an LD-set of $G$ with $|S| = |S'| \le \frac{n}{2}$. This contradicts our assumption that $\LD(G) > \frac{n}{2}$. Hence, $G$ cannot contain any $(F_2;\hat{u},\hat{v},\hat{x},\hat{y},\hat{z})$-subgraph.
\end{proof}

\begin{lemma}\label{lem-F3}
	If $G\in\calG$ has the smallest possible number of edges $m\geq7$ such that $\LD(G)>\frac{n}{2}$, then $G$ has no $(F_3;\hat{u},\hat{v},\hat{x},\hat{y},\hat{z})$-subgraphs.
\end{lemma}
\begin{proof} 
On the contrary, suppose that $G$ has an $(F_3;\hat{u},\hat{v},\hat{x},\hat{y},\hat{z})$-subgraph, say $H_3$. Applying the drop-hat naming convention on $V(H_3)$, the vertices $u$ and $v$ are open twins of degree~$3$ and $y$ and $z$ are closed twins of degree~$3$ in $G$.

We first show that either we can assume $m \ge 11$, or else, we end up with a contradiction. Let $G^* = G - \{u,v,x,y,z\}$. If $m \le 10$, then $m(G^*) \le 2$. In other words, $n(G^*) \le 3$. If $n(G^*) = 1$, that is, $V(G^*) = \{x'\}$, then the graph $G$ is determined on $n=6$ vertices to be isomorphic to $F_3$. Moreover, it can be verified that the set $S = \{x',v,z\}$ is an LD-set of $G$ such that $|S| = 3 = \frac{n}{2}$. This contradicts our assumption that $\LD(G) > \frac{n}{2}$. Hence, let us now assume that $n(G^*) = 2$ and that $V(G^*) = \{x',x''\}$. Then, we must have $x'x'' \in E(G)$ and, again, the graph $G$ is determined on $n=7$ vertices. In this case too, it can again be verified that the set $S = \{x',v,z\}$ is an LD-set of $G$ such that $|S| = 3 < \frac{n}{2}$. This again contradicts our assumption that $\LD(G) > \frac{n}{2}$. Hence, we now assume that $n(G^*) = 3$ and that $V(G^*) = \{x',x'',x'''\}$. If $x'x''' \in E(G)$, then in order for $x''$ and $x'''$ to not be open twins of degree~$1$ (since $G$ does not have open twins of degree~$1$), there must be an edge in $G^*$ other than $x'x''$ and $x'x'''$. Therefore, $m(G) \ge 11$ in this case. Thus, let $x'x''' \notin E(G)$ which implies that $x''x''' \in E(G)$. Then the graph $G$ is determined on $n=8$ vertices and it can  be verified that the set $S = \{x',x'',v,z\}$ is an LD-set of $G$ such that $|S| = 4 = \frac{n}{2}$. This again contradicts our assumption that $\LD(G) > \frac{n}{2}$. Hence, we may assume that $m \ge 11$.

Now, let $D = \{ux,uz,yz\}$ and $G' = G-D$. Then, we have $G' \in \calG$. Moreover, $8 \le m(G') < m$ and hence, by the minimality of $G$, there exists an LD-set $S'$ of $G'$ such that $|S'| \le \frac{n}{2}$. Notice that the vertices $u$ and $z$ are leaves with support vertices $y$ and $v$, respectively, in the graph $G'$. Therefore, by Lemmas~\ref{LemLDSup} and~\ref{LemLDLeaves}, we assume that $v,y \in S'$ and that $u,z \notin S'$. We now claim that $S'$ is also an LD-set of $G$. To prove so, we only need to show that, in the graph $G$, the vertices $u$ and $z$ are separated by $S'$ from all vertices in $V(G) \setminus S'$. The vertex $y\in S'$ separates $u$ and $z$ from other vertices in $V(G) \setminus S'$ and vertex $v\in S'$ separates $u$ and $z$ from each other. Hence, $S'$ is an LD-set of $G$ with $|S'| \le \frac{n}{2}$. This contradicts our assumption that $\LD(G) > \frac{n}{2}$ and proves that $G$ has no $(F_3;\hat{u},\hat{v},\hat{x},\hat{y},\hat{z})$-subgraphs.
\end{proof}

\begin{lemma}\label{lem-F4F5}
	If $G\in\calG$ has the smallest possible number of edges $m\geq7$ such that $\LD(G)>\frac{n}{2}$, then $G$ has neither $(F_4;\hat{u},\hat{v},\hat{w},\hat{x},\hat{y},\hat{z})$- nor $(F_5;\hat{u},\hat{v},\hat{w},\hat{x},\hat{y},\hat{z})$-subgraphs.
\end{lemma}
\begin{proof}
On the contrary, suppose that $G$ has an $(F_4;\hat{u},\hat{v},\hat{w},\hat{x},\hat{y},\hat{z})$-subgraph, say $H_4$. Then, in the drop-hat naming convention on $V(H_4)$, the pairs $u,v$ and $y,z$ are open twins of degree~$3$ in $G$. Now, if $\deg_G(w') = \deg_G(x') = 1$, then the graph $G$ is determined on $n=8$ vertices to be isomorphic to $F_4$. In this case, it can be verified that the set $S = \{v,w,x,y\}$ is an LD-set of $G$ with $|S| = 4 = \frac{n}{2}$. Hence, we may assume that the degree of $w'$ or $x'$ is at least $2$ in $G$. Therefore, we have $m \ge 11$. Now, let $D = \{ux,uy,wy\}$ and $G' = G-D$. Then, we have $G' \in \calG$. Moreover, $8 \le m(G') < m$ and hence, by the minimality of $G$, there exists an LD-set $S'$ of $G'$ such that $|S'| \le \frac{n}{2}$. 

We remark that the following arguments will also be used in the case of $(F_5;\hat{u},\hat{v},\hat{w},\hat{x},\hat{y},\hat{z})$-subgraph, that is when $w'=x'$, in the following paragraph. Notice that the vertices $u$ and $y$ are leaves with support vertices $z$ and $v$, respectively, in the graph $G'$. Hence, by Lemmas~\ref{LemLDSup} and~\ref{LemLDLeaves}, we assume that $v,z \in S'$ and that $u,y \notin S'$. We now claim that $S'$ is also an LD-set of $G$. To prove so, we only need to show that, in $G$, the vertices $u$ and $y$ are separated by $S'$ from all other vertices in $V(G) \setminus S'$. First of all, $z$ separates $u$ from each other vertex in $V(G) \setminus S'$ except possibly $w$. However, since $S'$ is locating-dominating in $G'$, we have $\{w,w'\}\cap S'\neq \emptyset$ and $w'$ separates $u$ and $w$. Similarly, $v$ separates $y$ from  $V(G) \setminus S'$ except possibly $x$. However, again either $x\in S'$ or $x'$ separates $y$ and $x$. This implies that $S'$ is an LD-set of $G$ with $|S'| \le \frac{n}{2}$. This contradicts our assumption that $\LD(G) > \frac{n}{2}$. Hence, $G$ cannot have an $(F_4;\hat{u},\hat{v},\hat{w},\hat{x},\hat{y},\hat{z})$-subgraph.

Again, on the contrary, suppose that $G$ has an $(F_5;\hat{u},\hat{v},\hat{w},\hat{x},\hat{y},\hat{z})$-subgraph, say $H_5$. Then, in the drop-hat naming convention on $V(H_5)$, the pairs $u,v$ and $y,z$ are open twins of degree~$3$ in $G$. Now, if $\deg_G(x') = 2$, then the graph $G$ is determined on $n=7$ vertices to be isomorphic to $F_5$. In this case, it can be verified that the set $S = \{v,x',y\}$ is an LD-set of $G$ with $|S| = 3 < \frac{n}{2}$. Hence, we may assume that $\deg_G(x') = 3$. Therefore, we have $m \ge 11$. Now, again let $D = \{ux,uy,wy\}$ and $G' = G-D$. Then again, $G' \in \calG$. Moreover, $8 \le m(G') < m$ and hence, $\LD(G') \leq \frac{n}{2}$. Notice that in the preceding arguments we did not require the case $w' = x'$ to be considered. This implies that by the exact same arguments as above, it can be shown that $G$ cannot have an $(F_5;\hat{u},\hat{v},\hat{w},\hat{x},\hat{y},\hat{z})$-subgraph. 
\end{proof}

\begin{lemma}\label{lem-F6}
	If $G\in\calG$ has the smallest possible number of edges $m\geq7$ such that $\LD(G)>\frac{n}{2}$, then $G$ has no $(F_6;\hat{u},\hat{v},\hat{x},\hat{y},\hat{z})$-subgraphs.
\end{lemma}
\begin{proof}
On the contrary, suppose that $G$ has an $(F_6;\hat{u},\hat{v},\hat{x},\hat{y},\hat{z})$-subgraph, say $H_6$. Then, applying the drop-hat naming convention on $V(H_6)$, the vertices $u$ and $v$ are open twins of degree~$3$ in~$G$. Now, if all three of $x',y'$ and $z'$ are leaves in $G$, then the graph $G$ is determined on $n=8$ vertices and it can be verified that the set $S = \{u,x,y,z\}$ is an LD-set of $G$ with $|S| = \frac{n}{2}$. This contradicts our assumption that $\LD(G) > \frac{n}{2}$. Therefore, without loss of generality, let us assume that $\deg_G(z') \ge 2$ and that $N_G(z') \setminus \{z\} = \{z''\}$ if $\deg_G(z') = 2$ and $N_G(z') \setminus \{z\} = \{z'',z'''\}$ if $\deg_G(z') = 3$. In what follows, we simply assume that $\deg_G(z')=3$ and that $N_G(z') \setminus \{z\} = \{z'',z'''\}$, as the arguments remain intact even when $\deg_G(z') = 2$ and the vertex $z'''$ is absent.

\setcounter{case}{0}

\begin{case}[The graph $G' = G-\{zz'\}$ does not contain open twins of degree~$1$ or $2$]

		In this case, let $D = \{zz'\}$ and $G' = G-D$. Let $F'_{z'}$ be the component of $G'$ to which the vertex $z'$ belongs and let $F'_{z}$ be the component of $G'$ to which the vertex $z$ belongs. Notice that we may possibly have $F'_{z'} = F'_z$. Furthermore, we have $F'_z, F'_{z'}\in \calG$. Moreover, we have $8 \le m(F'_z) < m$ and hence, by the minimality of $G$, there exists an LD-set $S'_{z}$ of the component $F'_{z}$ such that $|S'_{z}| \le \frac{n'_z}{2}$, where $n'_z$ is the order of the component $F'_z$. For the component $F'_{z'}$, we will denote by $S'_{z'}$ its minimum-sized LD-set.
		
		\begin{subcase}[$F'_{z'} \ne F'_z$]
			
			We show that there exists an LD-set $S'_{z'}$ of the component $F'_{z'}$ such that $|S'_{z'}| \le \frac{n'_{z'}}{2}$, where $n'_{z'}$ is the order of the component $F'_{z'}$. Let us first assume that $F'_{z'}$ does not contain any open twins of degree~$3$. Here we show that $F'_{z'}$ cannot be isomorphic to either $K_3$ or $K_4$. First of all, we notice that $F'_{z'}$ is not isomorphic to $K_4$ since the latter is $3$-regular and $\deg_{F'_{z'}}(z') \le 2$. Let us, therefore, assume that $F'_{z'} \cong K_3$. Thus, let $V(F'_{z'}) = \{z',z'',z'''\}$. Then, we take $D' = \{z'z'',z'z'''\}$ and $G'' = G-D'$. Then, the vertex $z'$ is a leaf with support vertex $z$ in a component, say $F''$, of $G''$ which belongs to $\calG$ and with $9 \le m(F'') < m$. Hence, by the minimality of $G$, there exists an LD-set $S''$ of $G''$ such that $|S''| \le \frac{n}{2}-1$. By Lemmas~\ref{LemLDSup} and~\ref{LemLDLeaves}, we assume that $z \in S''$ and $z' \notin S''$. Then, it can be verified that the set $S = S'' \cup \{z''\}$ is an LD-set of $G$ with $|S| \le \frac{n}{2}$. This implies a contradiction to our assumption that $\LD(G) > \frac{n}{2}$. Hence, $F'_{z'}$ is not isomorphic to $K_3$ either. This implies, by Theorem~\ref{theclosedtwins}, that there exists an LD-set $S'_{z'}$ of the component $F'_{z'}$ such that $|S'_{z'}| \le \frac{n'_{z'}}{2}$. Let us then assume that $F'_{z'}$ has a pair of open twins of degree~$3$. Since we have restricted the possible $(F;U)$-subgraphs of $G$ in previous claims, the component  $F'_{z'}$ together with vertex $z$ and edge $zz'$ must contain an $(F_6;\hat{u},\hat{v},\hat{x},\hat{y},\hat{z})$-subgraph. This implies that $8 \le m(F'_{z'}) < m$ and hence, by the minimality of $G$, there exists an LD-set $S'_{z'}$ of the component $F'_{z'}$ such that $|S'_{z'}| \le \frac{n'_{z'}}{2}$.
			
			We now claim that the set $S = S'_{z'} \cup S'_z$ is an LD-set of $G$. It can be verified that $S$ is a dominating set of $G$. To prove that $S$ is also a separating set of $G$, we only need to show that the vertex $z$ is separated by $S$ from all vertices in $\{z',z'',z'''\} \setminus S$ and the vertex $z'$ is separated by $S$ from the vertices in $\{u,v\} \setminus S$. However, the first of these claims is true due to the fact that $\{u,v\} \cap S'_z \ne \emptyset$ since $u$ and $v$ are open twins in the component $F'_z$; and second one holds by the fact that $\{z',z'',z'''\} \cap S'_{z'} \ne \emptyset$ in order for $S'_{z'}$ to dominate $z'$. Hence, $S$ is, indeed, an LD-set of $G$. Moreover, $|S| = |S'_{z'}| + |S'_z| \le \frac{n'_{z'}}{2} + \frac{n'_z}{2} = \frac{n}{2}$ contradicts our assumption that $\LD(G) > \frac{n}{2}$. \hfill $\blacktriangleleft$
		\end{subcase}
		
		\begin{subcase}[$F'_{z'} = F'_z = G'$]
			Let $S'_{z'} = S'_z = S'$. In this case, the set $S'$ is an LD-set of $G$ if either both $z,z' \in S'$ or both $z,z' \notin S'$. Since $u$ and $v$ are open twins in $G'$, it implies that $\{u,v\} \cap S' \ne \emptyset$. Therefore, without loss of generality, let us assume that $v \in S'$. Let us first assume that $z \in S'$ and $z' \notin S'$. Now, if on the contrary, $S'$ is not an LD-set of $G$, it implies that, the vertex $z'$ is not separated by $S$ from a vertex $p \in N_G(z) \setminus (S' \cup \{z'\})=\{u\}$ in the graph $G$. Therefore, we must have $\{z'',z'''\} \cap S' \ne \emptyset$ in order for $S'$ to dominate the vertex $z'$. Without loss of generality, let us assume that $z'' \in S'$. Since $z'$ and $u$ are not separated, we have $z''\in \{x,y\}$. Thus, $z'$ is adjacent to $x$ or $y$, contradicting the structure implied by $(F_6;\hat{u},\hat{v},\hat{x},\hat{y},\hat{z})$-subgraph. Thus, $S$ separates the pair $u,z'$ and is thus an LD-set of $G$ if $z \in S'$ and $z' \notin S'$.
			
			Let us next assume that $z' \in S'$ and $z \notin S'$. Again, let us assume on the contrary that $S'$ is not an LD-set of $G$. Recall that $v\in S'$. This implies that $S'$ does not separate the vertex $z$ and another vertex $p \in (N_G(v)\cap N_G(z')) \setminus (S' \cup \{z\})$. However, since $z'$ is not adjacent to $x$ or $y$, we have $N_G(v)\cap N_G(z')=\{z\}$. Hence, we cannot select $p$. This implies that $S$ is an LD-set of $G$ also if $z' \in S'$ and $z \notin S'$. Moreover, $|S'| \le \frac{n}{2}$ contradicts our assumption that $\LD(G) > \frac{n}{2}$. \hfill $\blacktriangleleft$
		\end{subcase}
		
		Hence, $G$ has no $(F_6;\hat{u},\hat{v},\hat{x},\hat{y},\hat{z})$-subgraphs in this case.
	\end{case}
	
	\begin{case}[The graph $G' = G-\{zz'\}$ has open twins of degree~$1$ or $2$]
	
		In this case, the vertex $z'$ is an open twin of degree~$1$ or $2$ with some vertex, say $z^*$, such that $N_G(z') \setminus \{z\} = N_G(z^*) \subseteq \{z'',z'''\}$. Therefore, we have $1 \le \deg_G(z^*) \le 2$. Let $D = \{z'z'',z'z'''\}$ and $G^* = G-D$. Moreover, let $F'_{z''}$ be the component of $G^*$ to which the vertex $z''$ belongs and let $F'_{z}$ be the component of $G^*$ to which the vertex $z$ (and also $z'$) belongs. Notice that we may possibly have $F'_{z''} = F'_z$.  Further notice that $F'_z,F'_{z''} \in \calG$. Indeed, the only candidates for open twins of degrees 1 or 2 are $z',z''$ and $z'''$. However, $z'$ is a leaf adjacent to $z$ which does not have other adjacent leaves. Moreover, $z^*$ is a neighbour only to the vertices 
		$z''$ and $z'''$ in $F'_{z'}$. This implies that neither $z''$ nor $z'''$ is an open twin with any vertices in $V(F'_{z''}) \setminus \{z'',z'''\}$. Furthermore, $z''$ and $z'''$ cannot be open twins of degree 1 since then they would be open twins of degree two in $G$. Finally, they cannot be open twins of degree 2 in $G^*$, since then they would be open twins of degree 3 in $G$ adjacent to vertex $z^*$ of degree 2 contradicting Claim 1 or 2. We denote by $S'_z$ and $S''_{z''}$ a minimum-sized locating-dominating set of $F'_z$ and $F'_{z''}$, respectively.
		
		\begin{subcase}[$F'_{z''} \ne F'_z$]
		
			To begin with, the component $F'_z$ of $G^*$ belongs to $\calG$. Moreover, we have $9 \le m(F'_z) < m$ and hence, by the minimality of $G$, we have $|S'_{z}| \le \frac{n'_z}{2}$, where $n'_z$ is the order of the component $F'_z$.

			We next show that $|S'_{z''}| \le \frac{n'_{z''}}{2}$, where $n'_{z''}$ is the order of the component $F'_{z''}$. First of all, the component $F'_{z''} \in \calG$. Let us first assume that $F'_{z''}$ does not contain any open twins of degree~$3$. Here we show that $F'_{z''}$ cannot be isomorphic to either $K_3$ or $K_4$. We notice that $F'_{z''}$ is not isomorphic to $K_4$ since the latter is $3$-regular and $\deg_{F'_{z''}}(z'') \le 2$. Let us, therefore, assume that $F'_{z''} \cong K_3$. Thus, $V(F'_{z''}) = \{z^*,z'',z'''\}$ and $z''z''' \in E(F'_{z''})$. Then, we take $D' = \{z''z^*,z''z'''\}$ and $G'' = G-D'$. Then, the graph $G'' \in \calG$ with $12 \le m(G'') < m$ and hence, by the minimality of $G$, there exists an LD-set $S''$ of $G''$ such that $|S''| \le \frac{n}{2}$. Moreover, notice that the vertices $z''$ and $z^*$ are leaves with support vertices $z'$ and $z'''$, respectively, in $G''$. Therefore, by Lemmas~\ref{LemLDSup} and~\ref{LemLDLeaves}, we assume that $z',z''' \in S''$ and $z'',z^* \notin S''$. Then, the set $S''$ is also an LD-set of $G$ since $z''$ is the only vertex in $V(G) \setminus S''$ with the $I$-set $I_{G}(S'';z'')=\{z',z'''\}$. Moreover, $|S''| \le \frac{n}{2}$ implies a contradiction to our assumption that $\LD(G) > \frac{n}{2}$. Hence, $F'_{z''}$ is not isomorphic to $K_3$ either. This implies, by Theorem~\ref{theclosedtwins}, that there exists an LD-set $S'_{z''}$ of the component $F'_{z''}$ such that $|S'_{z''}| \le \frac{n'_{z''}}{2}$. 
			
			Let us then assume that $F'_{z''}$ has a pair of open twins of degree~$3$. Then by the previous claims together with the fact that neither $z'',z'''$ nor $z^*$ can be open twins of degree $3$ in $G^*$, the component $F'_{z''}$ must contain an $(F_6;\hat{u},\hat{v},\hat{x},\hat{y},\hat{z})$-subgraph. This implies that $8 \le m(F'_{z''}) < m$ and hence, by the minimality of $G$, there exists an LD-set $S'_{z''}$ of the component $F'_{z''}$ such that $|S'_{z''}| \le \frac{n'_{z''}}{2}$.
			
			We now claim that the set $S = S'_{z''} \cup S'_z$ is an LD-set of $G$. It can be verified that $S$ is a dominating set of $G$. We notice that $z'$ is a leaf with support vertex $z$ in the component $F'_z$. Therefore, by Lemmas~\ref{LemLDSup} and~\ref{LemLDLeaves}, we have $z \in S$ and $z' \notin S$. Thus, to show that $S$ is also a separating set of $G$, we only need to show that the vertex $z'$ is separated by $S$ from all vertices in $(N_G[z''] \cup N_G[z''']) \setminus S$. However, this is true due to the fact that $z \in S$. Hence, $S$ is, indeed, an LD-set of $G$. Moreover, $|S| = |S'_{z''}| + |S'_z| \le \frac{n'_{z''}}{2} + \frac{n'_z}{2} \le \frac{n}{2}$ contradicts our assumption that $\LD(G) > \frac{n}{2}$. \hfill $\blacktriangleleft$
		\end{subcase}
		
		\begin{subcase}[$F'_{z'} = F'_z = G^*$]
			Then, let $S'_{z''} = S'_z = S'$. By Lemmas~\ref{LemLDSup} and~\ref{LemLDLeaves}, we have $z \in S$ and $z' \notin S$. Hence, the set $S'$ is an LD-set of $G$ if both $z'',z''' \notin S'$. Therefore, without loss of generality, let us assume that $z'' \in S'$. Now, if on the contrary, $S'$ is not an LD-set of $G$, it implies that the vertex $z'$ is not separated by $S$ from a vertex $p \in (N_G(z'')\cap N_G(z)) \setminus  \{z'\}=\emptyset$ in the graph $G$. Since $p$ cannot exist, $S'$ is an LD-set of $G$ with $|S'| \le \frac{n}{2}$ which contradicts our assumption that $\LD(G) > \frac{n}{2}$. \hfill $\blacktriangleleft$
		\end{subcase}
		
		Therefore, $G$ has no $(F_6;\hat{u},\hat{v},\hat{x},\hat{y},\hat{z})$-subgraphs in this case as well.
	\end{case}
	
	This proves the claim that $G$ has no $(F_6;\hat{u},\hat{v},\hat{x},\hat{y},\hat{z})$-subgraphs exhausting all possibilities in Cases 1 and 2.
\end{proof}

With the help of above lemmas, we are now ready to proof the following theorem.

\begin{theorem} \label{thm_deg3}
		Let $G$ be a connected subcubic graph on $m \ge 7$ edges, not isomorphic to $K_{3,3}$ and  without open twins of degree~$1$ or $2$. Then, we have
		\[
		\LD(G) \le \frac{n}{2}.
		\]
\end{theorem}

\begin{proof}
Let $n = n(G)$ be the number of vertices and $m = m(G)$ be the number of edges of the graph $G$. Since $m \ge 7$, the graph $G$ is not isomorphic to either $K_3$ or $K_4$. Therefore, if $G$ does not contain any open twins of degree~$3$, then the result holds by Theorem~\ref{theclosedtwins}. Hence, we assume from now on that $G$ has at least one pair of open twins of degree~$3$. In other words, $G$ has an $(F_0;\hat{u},\hat{v})$-subgraph. 
Notice that in a subcubic graph not isomorphic to $K_{3,3}$, for each vertex $u$ that is a twin of degree 3, there exists exactly one other vertex $v$ with the same open neighbourhood. Let us assume that $G\in\calG$, $\LD(G)>\frac{n}{2}$ and among those graphs $G$ has the smallest number of edges $m\geq7$.

We next consider the graph $F'_3 = F_3 - \{\hat{x}'\} \in \calG$ that has one pair of open twins of degree~$3$, namely $\hat{u}$ and $\hat{v}$. Now, it can be verified that the set $S'_3 = \{\hat{v}, \hat{z}\}$ is an LD-set of $F'_3$ with $|S'_3| < \frac{n}{2}$. Notice that $F'_3$ is the smallest graph (with respect to the number of both vertices and edges) with open twins of degree~$3$ but without open twins of degree~$1$ or~$2$ other than $K_{3,3}$.

Since, by assumption, $G$ has an $(F_0;\hat{u},\hat{v})$-subgraph, say $H_0$, under an injective homomorphism $j_0 : V(F_0) \to V(G)$, applying the  drop-hat naming convention, the vertices $j_0(\hat{u})$, $j_0(\hat{v})$, $j_0(\hat{x})$, $j_0(\hat{y})$ and $j_0(\hat{z})$ of $H_0$ are called $u$, $v$, $x$, $y$ and $z$. Since $G$ does not have open twins of degree~$2$, at least two of the vertices in $\{x,y,z\}$ must have degree~$3$ in $G$. Therefore, without loss of generality, let us assume that $\deg_G(y) = \deg_G(z) = 3$. Then, let $x'$ (if it exists), $y'$ and $z'$ be the neighbours of $x$, $y$ and $z$, respectively, in $V(G) \setminus \{u,v\}$. Notice that we may possibly have $y' = z'$. Then the following two cases arise.

\setcounter{case}{0}

\begin{case}[$\deg_G(x) = 2$]
	In this case, if $y' = z$, or equivalently, $z' = y$, this implies that $yz \in E(G)$ and therefore, $y$ and $z$ are closed twins of degree~$3$ in $G$. Since $\deg_G(x) = 2$, this implies that the graph $G$ is determined on $5$ vertices such that $G \cong F_3'= F_3 - \{\hat{x}'\}$. As we have seen, $\LD(G)=2<\frac{n}{2}$ in this case. Hence this possibility does not arise as it contradicts our assumption that $\LD(G) > \frac{n}{2}$. However, the following two other possibilities may arise.
	\begin{enumerate}
		\item $y' = z' = w$. In this case, if $\deg_G(w) = 2$ as well, then the graph $G$ is determined on $6$ vertices such that $G \cong F_1 - \{\hat{w}'\}$. Now, it can be verified that the set $S = \{u,z,w\}$ is an LD-set of $G$ with $|S| = \frac{n}{2}$. This contradicts our assumption that $\LD(G) > \frac{n}{2}$. Hence, we must have $\deg_G(w) = 3$. So, let $w'$ be the neighbour of $w$ in $V(G) \setminus \{y,z\}$. Now, we cannot have $w' = x$, since $\deg_G(x) = 2$. This implies that $G$ has an $(F_1;\hat{u},\hat{v},\hat{w},\hat{x},\hat{y},\hat{z})$-subgraph and we obtain a contradiction by Lemma~\ref{lem-F1}.
		\item $y' \ne z'$. In this case, $G$ contains an $(F_2;\hat{u},\hat{v},\hat{x},\hat{y},\hat{z})$-subgraph and we obtain a contradiction by Lemma~\ref{lem-F2}.
	\end{enumerate}
\end{case}

\begin{case}[$\deg_G(x) = 3$]
	In this case, if $x'=y'=z'$, then the graph $G$ is isomorphic to $K_{3,3}$ contradicting our assumption. Hence, this possibility cannot arise. However, any two of $x'$, $y'$ and $z'$ could be equal. This implies the following possibilities.
	\begin{enumerate}
		\item $\{x,y,z\} \cap \{x',y',z'\} \ne \emptyset$. Without loss of generality, let us assume that $y'=z$, or equivalently, $y=z'$. This implies that $yz \in E(G)$ and hence, the graph $G$ has an $(F_3;\hat{u},\hat{v},\hat{x},\hat{y},\hat{z})$-subgraph and we obtain a contradiction by Lemma~\ref{lem-F3}.
		\item $\{x,y,z\} \cap \{x',y',z'\} = \emptyset$ and $|\{x',y',z'\}| =2$. Without loss of generality, let us assume that $y'=z'=w$, say. Observe that now $y$ and $z$ are open twins of degree 3. Hence, if $\deg_G(w)=2$, then by interchanging the names of the pairs $w,x$ and  by renaming $x'$ as $w'$ we end up in  $(F_1;\hat{u},\hat{v},\hat{w},\hat{x},\hat{y},\hat{z})$-subgraph of $G$ as in Case 1. Hence, (in the original notation) we let $\deg_G(w)=3$ and let $w'$ be the neighbour of $w$ in $V(G) \setminus \{u,v\}$. If $x' = w$, or equivalently, $w'=x$, it implies that $xw \in E(G)$. In other words, contrary to our assumption, $G$ is isomorphic to $K_{3,3}$. Hence, we must have $\{w,x\} \cap \{w',x'\} = \emptyset$. Now, if $x' \ne w'$, then $H$ is an $(F_4;\hat{u},\hat{v},\hat{w},\hat{x},\hat{y},\hat{z})$-subgraph; and if $x'=w'$, then $H$ is an $(F_5;\hat{u},\hat{v},\hat{w},\hat{x},\hat{y},\hat{z})$-subgraph. Hence, with this possibility, the graph $G$ either has an $(F_4;\hat{u},\hat{v},\hat{w},\hat{x},\hat{y},\hat{z})$-subgraph or an $(F_5;\hat{u},\hat{v},\hat{w},\hat{x},\hat{y},\hat{z})$-subgraph and we obtain a contradiction by Lemma~\ref{lem-F4F5}.
		\item $\{x,y,z\} \cap \{x',y',z'\} = \emptyset$ and $|\{x',y',z'\}| = 3$. In this case, there is an 
		 $(F_6;\hat{u},\hat{v},\hat{x},\hat{y},\hat{z})$-subgraph in the graph $G$ and we obtain a contradiction by Lemma~\ref{lem-F6}.
	\end{enumerate}
\end{case}

	This concludes the proof.
\end{proof}

In particular, our results imply that we may extend Conjecture~\ref{conj_Garijo} to all cubic graphs with the exception of $K_4$ and $K_{3,3}$.

\begin{corollary}\label{corCubic}
Let $G$ be a connected cubic graph other than $K_4$ or $K_{3,3}$. Then, $\LD(G)\leq \frac{n}{2}$.
\end{corollary}


\section{Examples}\label{SecExample}

In this section, we consider some constructions which show the tightness of our results. First we show that we cannot further relax the twin-conditions in Theorem~\ref{thm_deg3} for subcubic graphs.

\begin{proposition}\label{PropOpentwinEx}
There exists
\begin{enumerate}
    \item an infinite family of connected subcubic graphs with open twins of degree 1; and
    \item an infinite family of connected subcubic graphs with open twins of degree 2,
\end{enumerate} which have location-domination number over half of their order.
\end{proposition}
\begin{proof}
Consider first Claim 1). Let $G$ be a connected subcubic graph on $n=12k$ vertices, for $k\geq1$, as in Figure~\ref{Fig_deg1opentwins}. The graph consists of a path $P_v$ on $3k$ vertices named $v_1,\dots,v_{3k}$. Furthermore, to each vertex $v_i$ we attach a three vertex path on vertices $a_i,b_i,c_i$ from the middle vertex $b_i$. Let us construct a minimum size locating-dominating set $S$ for $G$. Since each $a_i$ and $c_i$ are open twins, we may assume without loss of generality that each $c_i\in S$. Furthermore, since each $a_i$ needs to be dominated, at least one of the vertices $a_i,b_i\in S$. If $b_i\in S$ and $a_i\not\in S$, then one of the vertices $v_{i-1},v_i,v_{i+1}$ is in $S$ to separate $v_i$ and $a_i$. If $b_i\not\in S$, then we still require one of the three vertices to dominate $v_i$. In other words, a minimum size locating-dominating set contains a dominating set of $P_v$ which contains at least $k$ vertices. Hence, we have $\gamma^{LD}(G)=|S|\geq k+6k=\frac{7n}{12}$. Note that this 
lower bound is actually tight for this family of graphs as we obtain the value with the shaded vertices in  Figure~\ref{Fig_deg1opentwins} which can be verified to form a locating-dominating set.


\begin{figure}[!htb]
\centering
\begin{tikzpicture}[
blacknode/.style={circle, draw=black!, fill=black!, thick},
whitenode/.style={circle, draw=black!, fill=white!, thick},
scale=0.5]
\tiny
\node[whitenode] (0) at (0,0) {};
\node[blacknode] (1) at (3,0) {};
\node[whitenode] (2) at (6,0) {}; \node () at ($(2)+(0,-0.8)$) {\small $v_{i-1}$};
\node[whitenode] (3) at (9,0) {}; \node () at ($(3)+(0,-0.8)$) {\small $v_i$};
\node[blacknode] (4) at (12,0) {}; \node () at ($(4)+(0,-0.8)$) {\small $v_{i+1}$};
\node[whitenode] (5) at (15,0) {};
\node[whitenode] (6) at (18,0) {};
\node[blacknode] (7) at (21,0) {};
\node[blacknode] (0') at (0,2) {};
\node[blacknode] (0'') at (1,3) {};
\node[whitenode] (0''') at (-1,3) {};

\node[blacknode] (1') at (3,2) {};
\node[blacknode] (1'') at (4,3) {};
\node[whitenode] (1''') at (2,3) {};

\node[blacknode] (2') at (6,2) {}; \node () at ($(2')+(1.1,-0.1)$) {\small $b_{i-1}$};
\node[blacknode] (2'') at (7,3) {}; \node () at ($(2'')+(0.2,0.8)$) {\small $c_{i-1}$};
\node[whitenode] (2''') at (5,3) {}; \node () at ($(2''')+(0.8,0.4)$) {\small $a_{i-1}$};

\node[blacknode] (3') at (9,2) {}; \node () at ($(3')+(0.8,-0.1)$) {\small $b_i$};
\node[blacknode] (3'') at (10,3) {}; \node () at ($(3'')+(0,0.8)$) {\small $c_i$};
\node[whitenode] (3''') at (8,3) {}; \node () at ($(3''')+(0.5,0.4)$) {\small $a_i$};

\node[blacknode] (4') at (12,2) {}; \node () at ($(4')+(1.1,-0.1)$) {\small $b_{i+1}$};
\node[blacknode] (4'') at (13,3) {}; \node () at ($(4'')+(0.2,0.8)$) {\small $c_{i+1}$};
\node[whitenode] (4''') at (11,3) {}; \node () at ($(4''')+(0.8,0.4)$) {\small $a_{i+1}$};

\node[blacknode] (5') at (15,2) {};
\node[blacknode] (5'') at (16,3) {};
\node[whitenode] (5''') at (14,3) {};

\node[blacknode] (6') at (18,2) {};
\node[blacknode] (6'') at (19,3) {};
\node[whitenode] (6''') at (17,3) {};

\node[blacknode] (7') at (21,2) {};
\node[blacknode] (7'') at (22,3) {};
\node[whitenode] (7''') at (20,3) {};
\draw[-, thick] (0) -- (1);
\draw[-, thick] (1) -- (2);
\draw[-, thick] (2) -- (3);
\draw[-, thick] (3) -- (4);
\draw[-, thick] (4) -- (5);
\draw[-, thick] (5) -- (6);
\draw[-, thick] (6) -- (7);
\draw[dashed, thick] (7) -- (23,0);
\draw[dashed, thick] (0) -- (-2,0);
\draw[-, thick] (0) -- (0');
\draw[-, thick] (0') -- (0'');
\draw[-, thick] (0') -- (0''');

\draw[-, thick] (1) -- (1');
\draw[-, thick] (1') -- (1'');
\draw[-, thick] (1') -- (1''');

\draw[-, thick] (2) -- (2');
\draw[-, thick] (2') -- (2'');
\draw[-, thick] (2') -- (2''');

\draw[-, thick] (3) -- (3');
\draw[-, thick] (3') -- (3'');
\draw[-, thick] (3') -- (3''');

\draw[-, thick] (4) -- (4');
\draw[-, thick] (4') -- (4'');
\draw[-, thick] (4') -- (4''');

\draw[-, thick] (5) -- (5');
\draw[-, thick] (5') -- (5'');
\draw[-, thick] (5') -- (5''');

\draw[-, thick] (6) -- (6');
\draw[-, thick] (6') -- (6'');
\draw[-, thick] (6') -- (6''');

\draw[-, thick] (7) -- (7');
\draw[-, thick] (7') -- (7'');
\draw[-, thick] (7') -- (7''');
\end{tikzpicture}
\caption{Example of a subcubic graph $G$ on $n=12k$ vertices containing open twins of degree~$1$ and for which $\LD(G)=\frac{7}{12}n$. The shaded vertices constitute a minimum LD-set.}\label{Fig_deg1opentwins}
\end{figure}
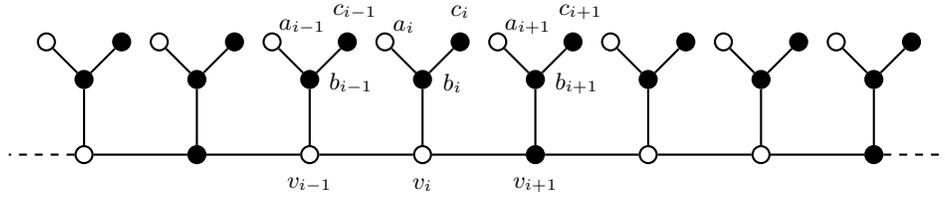


Consider next Claim 2).  Let $G$ be a connected subcubic graph on $n=60k$ vertices, for $k\geq1$, as in Figure~\ref{Fig_deg2opentwins}. The graph consists of a path $P_v$ on $5k$ vertices named $v_1,\dots,v_{5k}$ and to each vertex $v_i$, we connect an eleven vertex subgraph $G_i$ as in Figure~\ref{Fig_deg2opentwins}. Denote the vertex in $G_i$ which has an edge to $v_i$ by $u_i$. Figure~\ref{Fig_deg2opentwins} contains a minimum-sized locating-dominating set in shaded vertices. Notice that each subgraph $G_i$ has a pair of open twins of degree~$2$, we need to include one of them in any minimum-sized locating-dominating set $S$ of $G$. Furthermore, by Lemma~\ref{LemLDSup}, we may assume that the support vertex belongs to the set $S$. However, the support vertex itself is not enough to separate the leaf and the open-twin outside of $S$. Hence, $|S\cap (V(G_i)\setminus \{u_i\})|\geq 6$. Furthermore, these vertices do not dominate the vertices in the path $P_v$ which requires $2k$ vertices in any locating-dominating set (see~\cite{slater1988dominating}). Hence, we require at least $2k$ vertices in the set $S\cap (V(P_v)\cup\bigcup_{i=1}^{5k} \{u_i\})$. Therefore, we have $|S|\geq 6\cdot 5k+2k=32k=\frac{8n}{15}>\frac{n}{2}$.
\end{proof}


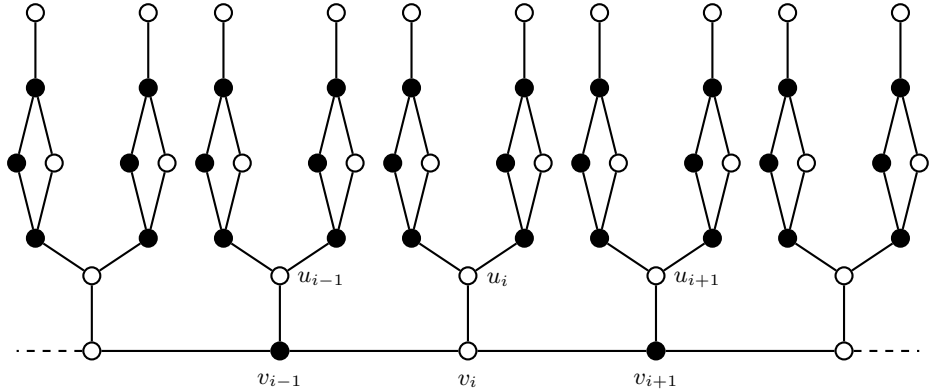
\begin{figure}[!htb]
\centering
\begin{tikzpicture}[
blacknode/.style={circle, draw=black!, fill=black!, thick},
whitenode/.style={circle, draw=black!, fill=white!, thick},
scale=0.5]
\tiny
\node[whitenode] (0) at (0,0) {};
\node[blacknode] (1) at (5,0) {}; \node () at ($(1)+(0,-0.8)$) {\small $v_{i-1}$};
\node[whitenode] (2) at (10,0) {}; \node () at ($(2)+(0,-0.8)$) {\small $v_i$};
\node[blacknode] (3) at (15,0) {}; \node () at ($(3)+(0,-0.8)$) {\small $v_{i+1}$};
\node[whitenode] (4) at (20,0) {};

\node[whitenode] (0') at (0,2) {};
\node[blacknode] (0'') at (1.5,3) {};
\node[blacknode] (0''') at (-1.5,3) {};
\node[whitenode] (0*) at (2,5) {};
\node[blacknode] (0**) at (1,5) {};
\node[blacknode] (0***) at (1.5,7) {};
\node[whitenode] (0****) at (1.5,9) {};
\node[whitenode] (0") at (-1,5) {};
\node[blacknode] (0"") at (-2,5) {};
\node[blacknode] (0""") at (-1.5,7) {};
\node[whitenode] (0"""") at (-1.5,9) {};

\node[whitenode] (1') at (5,2) {}; \node () at ($(1')+(1.1,-0.1)$) {\small $u_{i-1}$};
\node[blacknode] (1'') at (6.5,3) {};
\node[blacknode] (1''') at (3.5,3) {};
\node[whitenode] (1*) at (7,5) {};
\node[blacknode] (1**) at (6,5) {};
\node[blacknode] (1***) at (6.5,7) {};
\node[whitenode] (1****) at (6.5,9) {};
\node[whitenode] (1") at (4,5) {};
\node[blacknode] (1"") at (3,5) {};
\node[blacknode] (1""") at (3.5,7) {};
\node[whitenode] (1"""") at (3.5,9) {};

\node[whitenode] (2') at (10,2) {}; \node () at ($(2')+(0.8,-0.1)$) {\small $u_i$};
\node[blacknode] (2'') at (11.5,3) {};
\node[blacknode] (2''') at (8.5,3) {};
\node[whitenode] (2*) at (12,5) {};
\node[blacknode] (2**) at (11,5) {};
\node[blacknode] (2***) at (11.5,7) {};
\node[whitenode] (2****) at (11.5,9) {};
\node[whitenode] (2") at (9,5) {};
\node[blacknode] (2"") at (8,5) {};
\node[blacknode] (2""") at (8.5,7) {};
\node[whitenode] (2"""") at (8.5,9) {};

\node[whitenode] (3') at (15,2) {}; \node () at ($(3')+(1.1,-0.1)$) {\small $u_{i+1}$};
\node[blacknode] (3'') at (16.5,3) {};
\node[blacknode] (3''') at (13.5,3) {};
\node[whitenode] (3*) at (17,5) {};
\node[blacknode] (3**) at (16,5) {};
\node[blacknode] (3***) at (16.5,7) {};
\node[whitenode] (3****) at (16.5,9) {};
\node[whitenode] (3") at (14,5) {};
\node[blacknode] (3"") at (13,5) {};
\node[blacknode] (3""") at (13.5,7) {};
\node[whitenode] (3"""") at (13.5,9) {};

\node[whitenode] (4') at (20,2) {};
\node[blacknode] (4'') at (21.5,3) {};
\node[blacknode] (4''') at (18.5,3) {};
\node[whitenode] (4*) at (22,5) {};
\node[blacknode] (4**) at (21,5) {};
\node[blacknode] (4***) at (21.5,7) {};
\node[whitenode] (4****) at (21.5,9) {};
\node[whitenode] (4") at (19,5) {};
\node[blacknode] (4"") at (18,5) {};
\node[blacknode] (4""") at (18.5,7) {};
\node[whitenode] (4"""") at (18.5,9) {};

\draw[-, thick] (0) -- (1);
\draw[-, thick] (1) -- (2);
\draw[-, thick] (2) -- (3);
\draw[-, thick] (3) -- (4);

\draw[dashed, thick] (4) -- (22,0);
\draw[dashed, thick] (0) -- (-2,0);
\draw[-, thick] (0) -- (0');
\draw[-, thick] (0') -- (0'');
\draw[-, thick] (0') -- (0''');
\draw[-, thick] (0'') -- (0*);
\draw[-, thick] (0'') -- (0**);
\draw[-, thick] (0*) -- (0***);
\draw[-, thick] (0**) -- (0***);
\draw[-, thick] (0***) -- (0****);
\draw[-, thick] (0''') -- (0");
\draw[-, thick] (0''') -- (0"");
\draw[-, thick] (0") -- (0""");
\draw[-, thick] (0"") -- (0""");
\draw[-, thick] (0""") -- (0"""");

\draw[-, thick] (1) -- (1');
\draw[-, thick] (1') -- (1'');
\draw[-, thick] (1') -- (1''');
\draw[-, thick] (1'') -- (1*);
\draw[-, thick] (1'') -- (1**);
\draw[-, thick] (1*) -- (1***);
\draw[-, thick] (1**) -- (1***);
\draw[-, thick] (1***) -- (1****);
\draw[-, thick] (1''') -- (1");
\draw[-, thick] (1''') -- (1"");
\draw[-, thick] (1") -- (1""");
\draw[-, thick] (1"") -- (1""");
\draw[-, thick] (1""") -- (1"""");

\draw[-, thick] (2) -- (2');
\draw[-, thick] (2') -- (2'');
\draw[-, thick] (2') -- (2''');
\draw[-, thick] (2'') -- (2*);
\draw[-, thick] (2'') -- (2**);
\draw[-, thick] (2*) -- (2***);
\draw[-, thick] (2**) -- (2***);
\draw[-, thick] (2***) -- (2****);
\draw[-, thick] (2''') -- (2");
\draw[-, thick] (2''') -- (2"");
\draw[-, thick] (2") -- (2""");
\draw[-, thick] (2"") -- (2""");
\draw[-, thick] (2""") -- (2"""");

\draw[-, thick] (3) -- (3');
\draw[-, thick] (3') -- (3'');
\draw[-, thick] (3') -- (3''');
\draw[-, thick] (3'') -- (3*);
\draw[-, thick] (3'') -- (3**);
\draw[-, thick] (3*) -- (3***);
\draw[-, thick] (3**) -- (3***);
\draw[-, thick] (3***) -- (3****);
\draw[-, thick] (3''') -- (3");
\draw[-, thick] (3''') -- (3"");
\draw[-, thick] (3") -- (3""");
\draw[-, thick] (3"") -- (3""");
\draw[-, thick] (3""") -- (3"""");

\draw[-, thick] (4) -- (4');
\draw[-, thick] (4') -- (4'');
\draw[-, thick] (4') -- (4''');
\draw[-, thick] (4'') -- (4*);
\draw[-, thick] (4'') -- (4**);
\draw[-, thick] (4*) -- (4***);
\draw[-, thick] (4**) -- (4***);
\draw[-, thick] (4***) -- (4****);
\draw[-, thick] (4''') -- (4");
\draw[-, thick] (4''') -- (4"");
\draw[-, thick] (4") -- (4""");
\draw[-, thick] (4"") -- (4""");
\draw[-, thick] (4""") -- (4"""");
\end{tikzpicture}
\caption{Example of a subcubic graph $G$ on $n=60k$ vertices containing open twins of degree~$2$ and for which $\LD(G)=\frac{8}{15}n$. The shaded vertices constitute a minimum LD-set.}\label{Fig_deg2opentwins}
\end{figure}


The following proposition shows that the conjecture is not true in general for  $r$-regular graphs with twins. In other words, the result turns out to be a special property of cubic graphs.

\begin{proposition}\label{PropRegEx}
There exists an infinite family of connected $r$-regular graphs, for $r>3$, with closed twins which have location-domination number over half of their order.
\end{proposition}

\begin{proof}
Consider an $r$-regular graph $G_r$ on $n=(3r+3)k$ vertices for $r\geq4$ as in Figure~\ref{Fig_closedtwinsReg}. In particular, $G_r$ contains $3k$ copies of $(r-1)$-vertex cliques (ones within the 
dashed line in Figure~\ref{Fig_closedtwinsReg}). Vertices in such cliques are closed twins. We denote these cliques by $Q_1,\dots,Q_{3k}$. 
Furthermore, each clique $Q_i$ is adjacent to two vertices, let us denote these by $a_i$ and $b_i$ so that there is an edge $a_ib_{i-1}$ for $2 \le i \le 3k$ and $a_1b_{3k}$.

Let us consider a minimum-sized locating-dominating set $S$ for $G_r$. In particular, we have $|S\cap Q_i|\geq r-2$ for each $i$. Let us denote by $c_i\in Q_i$ the single vertex which might not be in $Q_i\cap S$. We note that set $S\cap Q_i$ does not separate vertices $a_i,b_i,c_i$. Let us assume that $|(Q_i\cup\{a_i,b_i\})\cap S|=r-2$ for some $i$. We observe that then we have $b_{i-1},a_{i+1}\in S$. Hence, we have $|(Q_{i-1}\cup\{a_{i-1},b_{i-1}\})\cup(Q_i\cup\{a_i,b_i\})\cup(Q_{i+1}\cup\{a_{i+1},b_{i+1}\})\cap S|\geq 3r-4$. Note that if we had $|(Q_i\cup\{a_i,b_i\})\cap S|\geq r-1$ for each $i$, then we would have more vertices in $S$. Hence, we have $|S|\geq\frac{3r-4}{3r+3}n$. When $r=4$ this gives $\LD(G)\geq \frac{8}{15}n>\frac{n}{2}$. We note that this lower bound is attainable with the construction used in Figure~\ref{Fig_closedtwinsReg}.  
\end{proof}


\begin{figure}[!htb]
\centering
\begin{tikzpicture}[
blacknode/.style={circle, draw=black!, fill=black!, thick},
whitenode/.style={circle, draw=black!, fill=white!, thick},
scale=0.5]
\tiny
%
\node[whitenode] (0) at (-2,0) {}; \node () at ($(0)+(0,0.8)$) {\small $a_1$};
\node[blacknode] (1) at (2,0) {}; \node () at ($(1)+(0,0.8)$) {\small $b_1$};
\node[whitenode] (2) at (4,0) {}; \node () at ($(2)+(0,0.8)$) {\small $a_2$};
\node[whitenode] (3) at (8,0) {}; \node () at ($(3)+(0,0.8)$) {\small $b_2$};
\node[blacknode] (4) at (10,0) {}; \node () at ($(4)+(0,0.8)$) {\small $a_3$};
\node[whitenode] (5) at (14,0) {}; \node () at ($(5)+(0,0.8)$) {\small $b_3$};
\node[blacknode] (6) at (18,0) {}; \node () at ($(6)+(0,0.8)$) {\small $a_{3k}$};
\node[whitenode] (7) at (22,0) {}; \node () at ($(7)+(0,0.8)$) {\small $b_{3k}$};

\node[blacknode] at (0,1.6) {};
\node[blacknode] at (0,0.6) {};
\node[blacknode] at (0,-0.4) {};
\node[whitenode] at (0,-1.4) {};

\node[blacknode] at (6,1.6) {};
\node[blacknode] at (6,0.6) {};
\node[blacknode] at (6,-0.4) {};
\node[whitenode] at (6,-1.4) {};

\node[blacknode] at (12,1.6) {};
\node[blacknode] at (12,0.6) {};
\node[blacknode] at (12,-0.4) {};
\node[whitenode] at (12,-1.4) {};

\node[blacknode] at (20,1.6) {};
\node[blacknode] at (20,0.6) {};
\node[blacknode] at (20,-0.4) {};
\node[whitenode] at (20,-1.4) {};
\node at (16,0) {$\mathbf{\cdots}$};
\draw[-, thick] (1) edge (2);
\draw[-, thick] (3) edge (4);
\draw[-, thick] (5) edge ($(5)!0.2!(6)$) edge[dotted] ($(5)!0.3!(6)$);
\draw[-, thick] (6) edge ($(6)!0.2!(5)$) edge[dotted] ($(6)!0.3!(5)$);
\draw[-, thick] (0) -- (-3,-1.5) .. controls (-4,-5) and (25,-5) .. (23,-1.5) -- (7);
\draw[thick, rounded corners] ($(0,0)+(-0.8,2.2)$) rectangle ($(0,0)+(0.8,-2.2)$) node at (0,3.4) {\small $(r-1)$-clique};
\draw[thick, rounded corners] ($(6,0)+(-0.8,2.2)$) rectangle ($(6,0)+(0.8,-2.2)$) node at (6,3.4) {\small $(r-1)$-clique};
\draw[thick, rounded corners] ($(12,0)+(-0.8,2.2)$) rectangle ($(12,0)+(0.8,-2.2)$) node at (12,3.4) {\small $(r-1)$-clique};
\draw[thick, rounded corners] ($(20,0)+(-0.8,2.2)$) rectangle ($(20,0)+(0.8,-2.2)$) node at (20,3.4) {\small $(r-1)$-clique};

\draw[thick, dashed, rounded corners] ($(0,0)+(-1.35,2.8)$) rectangle ($(0,0)+(1.35,-2.8)$);
\draw[thick, dashed, rounded corners] ($(6,0)+(-1.35,2.8)$) rectangle ($(6,0)+(1.35,-2.8)$);
\draw[thick, dashed, rounded corners] ($(12,0)+(-1.35,2.8)$) rectangle ($(12,0)+(1.35,-2.8)$);
\draw[thick, dashed, rounded corners] ($(20,0)+(-1.35,2.8)$) rectangle ($(20,0)+(1.35,-2.8)$);

\draw[-, thick] (0) edge ($(0)!0.6!(0,2)$) edge [dotted] ($(0)!0!(0,2)$);
\draw[-, thick] (0) edge ($(0)!0.35!(0,1)$) edge [dotted] ($(0)!0.55!(0,1)$);
\draw[-, thick] (0) edge ($(0)!0.35!(0,0)$) edge [dotted] ($(0)!0.53!(0,0)$);
\draw[-, thick] (0) edge ($(0)!0.35!(0,-1)$) edge [dotted] ($(0)!0.55!(0,-1)$);
\draw[-, thick] (0) edge ($(0)!0.6!(0,-2)$) edge [dotted] ($(0)!0!(0,-2)$);

\draw[-, thick] (1) edge ($(1)!0.6!(0,2)$) edge [dotted] ($(1)!0!(0,2)$);
\draw[-, thick] (1) edge ($(1)!0.35!(0,1)$) edge [dotted] ($(1)!0.55!(0,1)$);
\draw[-, thick] (1) edge ($(1)!0.35!(0,0)$) edge [dotted] ($(1)!0.53!(0,0)$);
\draw[-, thick] (1) edge ($(1)!0.35!(0,-1)$) edge [dotted] ($(1)!0.55!(0,-1)$);
\draw[-, thick] (1) edge ($(1)!0.6!(0,-2)$) edge [dotted] ($(1)!0!(0,-2)$);

\draw[-, thick] (2) edge ($(2)!0.6!(6,2)$) edge [dotted] ($(2)!0!(6,2)$);
\draw[-, thick] (2) edge ($(2)!0.35!(6,1)$) edge [dotted] ($(2)!0.55!(6,1)$);
\draw[-, thick] (2) edge ($(2)!0.35!(6,0)$) edge [dotted] ($(2)!0.53!(6,0)$);
\draw[-, thick] (2) edge ($(2)!0.35!(6,-1)$) edge [dotted] ($(2)!0.55!(6,-1)$);
\draw[-, thick] (2) edge ($(2)!0.6!(6,-2)$) edge [dotted] ($(2)!0!(6,-2)$);

\draw[-, thick] (3) edge ($(3)!0.6!(6,2)$) edge [dotted] ($(3)!0!(6,2)$);
\draw[-, thick] (3) edge ($(3)!0.35!(6,1)$) edge [dotted] ($(3)!0.55!(6,1)$);
\draw[-, thick] (3) edge ($(3)!0.35!(6,0)$) edge [dotted] ($(3)!0.53!(6,0)$);
\draw[-, thick] (3) edge ($(3)!0.35!(6,-1)$) edge [dotted] ($(3)!0.55!(6,-1)$);
\draw[-, thick] (3) edge ($(3)!0.6!(6,-2)$) edge [dotted] ($(3)!0!(6,-2)$);

\draw[-, thick] (4) edge ($(4)!0.6!(12,2)$) edge [dotted] ($(4)!0!(12,2)$);
\draw[-, thick] (4) edge ($(4)!0.35!(12,1)$) edge [dotted] ($(4)!0.55!(12,1)$);
\draw[-, thick] (4) edge ($(4)!0.35!(12,0)$) edge [dotted] ($(4)!0.53!(12,0)$);
\draw[-, thick] (4) edge ($(4)!0.35!(12,-1)$) edge [dotted] ($(4)!0.55!(12,-1)$);
\draw[-, thick] (4) edge ($(4)!0.6!(12,-2)$) edge [dotted] ($(4)!0!(12,-2)$);

\draw[-, thick] (5) edge ($(5)!0.6!(12,2)$) edge [dotted] ($(5)!0!(12,2)$);
\draw[-, thick] (5) edge ($(5)!0.35!(12,1)$) edge [dotted] ($(5)!0.55!(12,1)$);
\draw[-, thick] (5) edge ($(5)!0.35!(12,0)$) edge [dotted] ($(5)!0.53!(12,0)$);
\draw[-, thick] (5) edge ($(5)!0.35!(12,-1)$) edge [dotted] ($(5)!0.55!(12,-1)$);
\draw[-, thick] (5) edge ($(5)!0.6!(12,-2)$) edge [dotted] ($(5)!0!(12,-2)$);

\draw[-, thick] (6) edge ($(6)!0.6!(20,2)$) edge [dotted] ($(6)!0!(20,2)$);
\draw[-, thick] (6) edge ($(6)!0.35!(20,1)$) edge [dotted] ($(6)!0.55!(20,1)$);
\draw[-, thick] (6) edge ($(6)!0.35!(20,0)$) edge [dotted] ($(6)!0.53!(20,0)$);
\draw[-, thick] (6) edge ($(6)!0.35!(20,-1)$) edge [dotted] ($(6)!0.55!(20,-1)$);
\draw[-, thick] (6) edge ($(6)!0.6!(20,-2)$) edge [dotted] ($(6)!0!(20,-2)$);

\draw[-, thick] (7) edge ($(7)!0.6!(20,2)$) edge [dotted] ($(7)!0!(20,2)$);
\draw[-, thick] (7) edge ($(7)!0.35!(20,1)$) edge [dotted] ($(7)!0.55!(20,1)$);
\draw[-, thick] (7) edge ($(7)!0.35!(20,0)$) edge [dotted] ($(7)!0.53!(20,0)$);
\draw[-, thick] (7) edge ($(7)!0.35!(20,-1)$) edge [dotted] ($(7)!0.55!(20,-1)$);
\draw[-, thick] (7) edge ($(7)!0.6!(20,-2)$) edge [dotted] ($(7)!0!(20,-2)$);

\end{tikzpicture}
\caption{Example of an $r$-regular graph $G_r$ on $n=(3r+3)k$ vertices for which $\LD(G_r) = \frac{3r-4}{3r+3}n$. Therefore, for $r \ge 4$, we have $\LD(G_r) \ge \frac{8}{15}n$ which proves that Theorem~\ref{theclosedtwins} is not true for graphs of maximum degree greater than~$3$. The shaded vertices constitute a minimum LD-set.}\label{Fig_closedtwinsReg}
\end{figure}
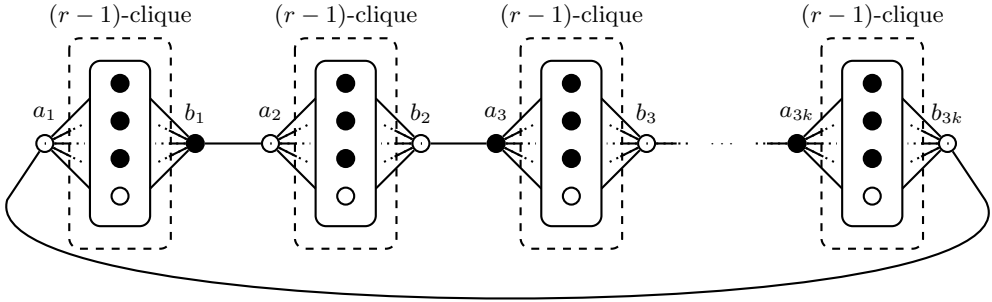


The following proposition shows that the bound in Conjecture~\ref{conj_Garijo} is tight for an infinite family of twin-free subcubic graphs.  We remark here that there also exists a simpler tight construction (than the one in the following proposition) for subcubic graphs, which is a path (or a cycle) with a leaf attached to each of its vertices. However, as there are no known tight constructions for cubic graphs and one could ponder whether the simple constructions are the only tight examples for subcubic graphs, we include the proof of the following more complicated family. In particular, this shows that the class of subcubic twin-free graphs requiring exactly half of their vertices in a locating-dominating set is relatively rich or at least non-trivial.  

\begin{proposition}\label{PropTightSubcubic}
There exists an infinite family of connected twin-free subcubic graphs which have location-domination number equal to half of their order.
\end{proposition}

\begin{proof}
Consider the graph $G$ on $n=8k+2$ vertices as in Figure~\ref{Fig_tightSubCub}. The graph consists of a path on $4k+1$ vertices. To each path vertex $p_i$, $1\leq i\leq 4k+1$, we join a new vertex $u_i$ and an edge from $u_i$ to $u_{i+1}$ if $i\equiv 2\mod4$ or $i\equiv 3\mod4$. Note that in particular $u_1$ and $u_{4k+1}$ are leaves. Consider next a minimum-sized LD-set $S$ in $G$. We note that for each pair $\{p_i,u_i\}$ where $u_i$ is a leaf, we have $\{p_i,u_i\}\cap S\neq \emptyset$ since $u_i$ is dominated by $S$. Let us next show that for each set $L_j=\{p_{j-1},p_j,p_{j+1},u_{j-1},u_j,u_{j+1}\}$ where the vertices contain a 6-cycle, we have $|L_j\cap S|\geq3$. Suppose on the contrary that $|L_j\cap S|\leq2$.
Assume first that $\{p_{j-1},u_{j-1}\}\cap S=\emptyset$. To dominate $u_{j-1}$, we have $u_j\in S$. The only single vertex that can separate both $p_j$ and $u_{j+1}$ from $u_{j-1}$ is $p_{j+1}$. However, vertices $u_j$ and $p_{j+1}$ cannot separate $p_j$ and $u_{j+1}$. Hence, the assumption $\{p_{j-1},u_{j-1}\}\cap S=\emptyset$ leads to $|L_j\cap S|\geq3$, a contradiction. Hence, by symmetry, we assume that $|\{p_{j-1},u_{j-1}\}\cap S|=1$ and $|\{p_{j+1},u_{j+1}\}\cap S|=1$ while $|\{p_{j},u_{j}\}\cap S|=0$. Notice that to dominate both $u_j$ and $p_j$ we have $L_j\cap S =\{p_{j-1},u_{j+1}\}$ or $L_j\cap S =\{p_{j+1},u_{j-1}\}$. However, the first of these options does not separate $u_{j-1}$ and $p_j$ while the second option does not separate $p_j$ and $u_{j+1}$. Hence, we have $|L_j\cap S|\geq3$. This implies that $\LD(G)\geq n/2$. By Proposition~\ref{ProbSubcubic}, we have $\LD(G)\leq n/2$. Therefore, $\LD(G)=n/2$.
\end{proof}


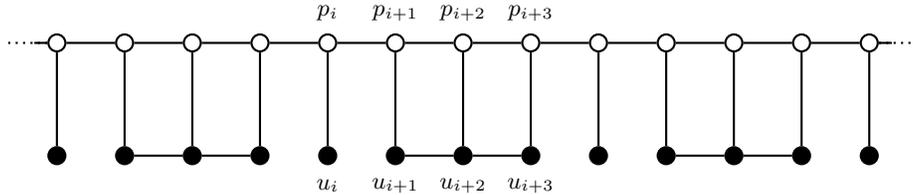
\begin{figure}[!htb]
\centering
\begin{tikzpicture}[
blacknode/.style={circle, draw=black!, fill=black!, thick},
whitenode/.style={circle, draw=black!, fill=white!, thick},
scale=0.5]
\tiny
\node[whitenode] (0) at (0,0) {};
\node[whitenode] (1) at (1.8,0) {};
\node[whitenode] (2) at (3.6,0) {};
\node[whitenode] (3) at (5.4,0) {};
\node[whitenode] (4) at (7.2,0) {}; \node () at ($(4)+(0,0.8)$) {\small $p_i$};
\node[whitenode] (5) at (9,0) {}; \node () at ($(5)+(0,0.8)$) {\small $p_{i+1}$};
\node[whitenode] (6) at (10.8,0) {}; \node () at ($(6)+(0,0.8)$) {\small $p_{i+2}$};
\node[whitenode] (7) at (12.6,0) {}; \node () at ($(7)+(0,0.8)$) {\small $p_{i+3}$};
\node[whitenode] (8) at (14.4,0) {};
\node[whitenode] (9) at (16.2,0) {};
\node[whitenode] (10) at (18,0) {};
\node[whitenode] (11) at (19.8,0) {};
\node[whitenode] (12) at (21.6,0) {};

\node[blacknode] (0') at ($(0)+(0,-3)$) {};

\node[blacknode] (1') at ($(1)+(0,-3)$) {};
\node[blacknode] (2') at ($(2)+(0,-3)$) {};
\node[blacknode] (3') at ($(3)+(0,-3)$) {};

\node[blacknode] (4') at ($(4)+(0,-3)$) {}; \node () at ($(4')+(0,-0.8)$) {\small $u_i$};

\node[blacknode] (5') at ($(5)+(0,-3)$) {}; \node () at ($(5')+(0,-0.8)$) {\small $u_{i+1}$};
\node[blacknode] (6') at ($(6)+(0,-3)$) {}; \node () at ($(6')+(0,-0.8)$) {\small $u_{i+2}$};
\node[blacknode] (7') at ($(7)+(0,-3)$) {}; \node () at ($(7')+(0,-0.8)$) {\small $u_{i+3}$};

\node[blacknode] (8') at ($(8)+(0,-3)$) {};

\node[blacknode] (9') at ($(9)+(0,-3)$) {};
\node[blacknode] (10') at ($(10)+(0,-3)$) {};
\node[blacknode] (11') at ($(11)+(0,-3)$) {};

\node[blacknode] (12') at ($(12)+(0,-3)$) {};

\draw[-, thick] (0) edge ($(0)!0.3!(-2,0)$) edge [dotted] ($(0)!0.7!(-2,0)$);
\draw[-, thick] (12) edge ($(12)!0.3!(23.6,0)$) edge [dotted] ($(12)!0.7!(23.6,0)$);

\draw[-, thick] (0) -- (1);
\draw[-, thick] (1) -- (2);
\draw[-, thick] (2) -- (3);
\draw[-, thick] (3) -- (4);
\draw[-, thick] (4) -- (5);
\draw[-, thick] (5) -- (6);
\draw[-, thick] (6) -- (7);
\draw[-, thick] (7) -- (8);
\draw[-, thick] (8) -- (9);
\draw[-, thick] (9) -- (10);
\draw[-, thick] (10) -- (11);
\draw[-, thick] (11) -- (12);

\draw[-, thick] (0) -- (0');
\draw[-, thick] (1) -- (1');
\draw[-, thick] (2) -- (2');
\draw[-, thick] (3) -- (3');
\draw[-, thick] (4) -- (4');
\draw[-, thick] (5) -- (5');
\draw[-, thick] (6) -- (6');
\draw[-, thick] (7) -- (7');
\draw[-, thick] (8) -- (8');
\draw[-, thick] (9) -- (9');
\draw[-, thick] (10) -- (10');
\draw[-, thick] (11) -- (11');
\draw[-, thick] (12) -- (12');

\draw[-, thick] (1') -- (2');
\draw[-, thick] (2') -- (3');

\draw[-, thick] (5') -- (6');
\draw[-, thick] (6') -- (7');

\draw[-, thick] (9') -- (10');
\draw[-, thick] (10') -- (11');

\end{tikzpicture}
\caption{Example of a subcubic graph $G$ on $n$ verices for which $\LD(G) = \frac{n}{2}$. The shaded vertices constitute a minimum LD-set.} \label{Fig_tightSubCub}
\end{figure}


\section{Conclusion}\label{SecConclusion}

In this article, we have proven Conjecture~\ref{conj_Garijo} for subcubic graphs and answered positively to Problems \ref{ProbSubcubic} and \ref{ProbCubicTwins}. In particular, we show that for each connected subcubic graph $G$, other than $K_1,K_4,K_{3,3}$ and without open twins of degrees $1$ or $2$, we have $\LD(G)\leq \frac{n}{2}$. We have also shown that these restrictions on open twins are necessary and that similar relaxation of conditions in Conjecture~\ref{conj_Garijo} is not possible for $r$-regular graphs in general.

Furthermore, we have presented an infinite family of twin-free subcubic graphs for which this bound is tight. However, the only known tight examples for the $n/2$-bound over connected twin-free cubic graphs are on six and eight vertices. Moreover, 
we were unable to find any such tight example for the $n/2$-bound by going through all connected twin-free cubic graphs on ten vertices. On the other hand, the $10$-vertex graph in Figure~\ref{Fig_tightCubic} is an example of a cubic graph containing both open and closed twins for which the conjectured upper bound is tight.
In \cite{foucaud2016location}, Foucaud and Henning  asked to characterize every twin-free cubic graph which attains the $n/2$-bound. We present a new open problem in the same vein:

\begin{oproblem}\label{OP1}
Does there exist an infinite family of connected (twin-free) cubic graphs which have LD-number equal to half of their order?
\end{oproblem}

Considering the previous open problem is interesting for both twin-free cubic graphs and graphs which allow twins. It would even be interesting if one could find a single connected twin-free cubic on at least ten vertices which has LD-number equal to half of its order.
If there does not exist any such (twin-free) cubic graphs, that also prompts another open problem:

\begin{oproblem}\label{OP2}
What is the (asymptotically) tight upper bound for LD-number of connected (twin-free) cubic graphs on at least ten vertices?
\end{oproblem}


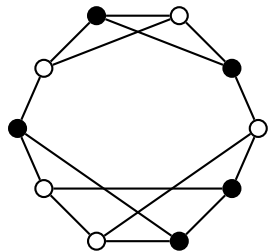
\begin{figure}[!htb]
\centering
\begin{tikzpicture}[
blacknode/.style={circle, draw=black!, fill=black!, thick},
whitenode/.style={circle, draw=black!, fill=white!, thick},
scale=0.5]
\tiny
\node[whitenode] (0) at (-1.1,0) {};
\node[blacknode] (1) at (1.1,0) {};
\node[blacknode] (2) at (2.5,1.4) {};
\node[whitenode] (3) at (3.2,3) {};
\node[blacknode] (4) at (2.5,4.6) {};
\node[whitenode] (5) at (1.1,6) {};
\node[blacknode] (6) at (-1.1,6) {};
\node[whitenode] (7) at (-2.5,4.6) {};
\node[blacknode] (8) at (-3.2,3) {};
\node[whitenode] (9) at (-2.5,1.4) {};

\draw[-, thick] (0) -- (1);
\draw[-, thick] (1) -- (2);
\draw[-, thick] (2) -- (3);
\draw[-, thick] (3) -- (4);
\draw[-, thick] (4) -- (5);
\draw[-, thick] (5) -- (6);
\draw[-, thick] (6) -- (7);
\draw[-, thick] (7) -- (8);
\draw[-, thick] (8) -- (9);
\draw[-, thick] (9) -- (0);

\draw[-, thick] (0) -- (3);
\draw[-, thick] (1) -- (8);
\draw[-, thick] (9) -- (2);
\draw[-, thick] (6) -- (4);
\draw[-, thick] (5) -- (7);

\end{tikzpicture}
\caption{Example of a cubic graph $G$ on $n=10$ vertices containing both open and closed twins for which $\LD(G) = \frac{n}{2}$. The shaded vertices constitute a minimum LD-set.} \label{Fig_tightCubic}
\end{figure}


\section*{Declarations}

 Some of the work was done during D. Chakraborty's visit to the Department of Mathematics and Statistics, University of Turku, Turku, Finland.
Some of the work of T.~Lehtilä was done during his stays at Department of Computer Science, University of Helsinki, Helsinki, Finland and Université Clermont-Auvergne, LIMOS, Clermont-Ferrand, France.

\section*{Funding}
This project was partially funded by the Research Council of Finland grant number 338797 and the ANR project GRALMECO (ANR-21-CE48-0004).
D.~Chakraborty was partially funded by the French government IDEX-ISITE initiative CAP 20-25 (ANR-16-IDEX-0001), the International Research Center ``Innovation Transportation and Production Systems" of the I-SITE CAP 20-25 and by the Lebanese American University under the President’s Intramural Research Fund PIRF0056.
T.~Lehtil\"a was partially funded by the Business Finland Project 6GNTF, funding decision 10769/31/2022 and by the Research Council of Finland grant number 358718.

\section*{Competing interests}
The authors declare that they have no conflicts of interest.

\section*{Data availability statement}
No data were used or generated in the current study.

\bibliographystyle{abbrv}

\bibliography{CubicLD}

\end{document}